\documentclass[12pt]{amsart}

\usepackage{amsfonts,amssymb,latexsym,amsmath,amscd}
\usepackage{mathrsfs}
\usepackage{times}
\usepackage{anysize}
\usepackage{dbnsymb}
\usepackage{graphicx}

\newif \ifhavetikz 
\havetikztrue

\ifhavetikz \usepackage{tikz} \fi

\marginsize{1in}{1in}{1in}{1in}

\renewcommand{\AA}{\mathbb{A}}

\newcommand{\CC}{\mathbb{C}}

\newcommand{\LL}{\mathbb{L}}

\newcommand{\NN}{\mathbb{N}}

\newcommand{\PP}{\mathbb{P}}
\newcommand{\QQ}{\mathbb{Q}}

\newcommand{\ZZ}{\mathbb{Z}}

\newcommand{\Pp}{\mathbf{P}}

\newcommand{\cC}{\mathcal{C}}

\newcommand{\fF}{\mathcal{F}}
\newcommand{\gG}{\mathcal{G}}
\newcommand{\hH}{\mathcal{H}}
\newcommand{\iI}{\mathcal{I}}
\newcommand{\jJ}{\mathcal{J}}

\newcommand{\oO}{\mathcal{O}}

\newcommand{\sS}{\mathcal{S}}
\newcommand{\tT}{\mathcal{T}}

\newcommand{\xX}{\mathcal{X}}

\newcommand{\zZ}{\mathcal{Z}}

\newcommand{\ii}{\mathfrak{i}}
\newcommand{\jj}{\mathfrak{j}}
\newcommand{\mm}{\mathfrak{m}}

\newcommand{\wtp}{\mathfrak{w}}

\newtheorem{lemma}{Lemma}
\newtheorem{conjecture}[lemma]{Conjecture}
\newtheorem{corollary}[lemma]{Corollary}
\newtheorem{theorem}[lemma]{Theorem}
\newtheorem{definition}[lemma]{Definition}

\newtheorem{proposition}[lemma]{Proposition}

\theoremstyle{remark} 
\newtheorem*{remark}{Remark}
\newtheorem*{example}{Example}

\DeclareMathOperator{\gr}{gr}
\DeclareMathOperator{\Hilb}{Hilb}

\DeclareMathOperator{\Hom}{Hom}

\DeclareMathOperator{\Exp}{Exp}
\DeclareMathOperator{\Spec}{Spec}
\DeclareMathOperator{\ord}{ord}

\newcommand{\uspp}{\overline{\mathscr{P}}}
\newcommand{\usppa}{\overline{\mathscr{P}}_{\mathrm{alg}}}
\newcommand{\spp}{\mathscr{P}}
\newcommand{\sppa}{\mathscr{P}_{\mathrm{alg}}}

\title[Hilbert schemes and HOMFLY homology]{The Hilbert scheme of a plane curve singularity \\ and the HOMFLY
homology of its link}

\author[A. Oblomkov, J. Rasmussen, and V. Shende; with E. Gorsky]{  Alexei Oblomkov  \and Jacob Rasmussen  \and  Vivek Shende \\ \\
with an appendix by Eugene Gorsky}

\begin{document}

\begin{abstract}
  We conjecture an expression for the dimensions of the Khovanov-Rozansky
  HOMFLY homology groups of the link of a plane curve singularity in terms
  of the weight polynomials of Hilbert schemes of points 
  scheme-theoretically supported
  on the singularity.  The conjecture specializes to our previous
  conjecture  \cite{OS} 
  relating the HOMFLY polynomial to the Euler numbers of the 
  same spaces upon setting $t = -1$.  
  By generalizing results of Piontkowski on the structure of
  compactified Jacobians to the case of Hilbert
  schemes of points, we give an explicit prediction 
  of the HOMFLY homology of a $(k,n)$ torus knot
  as a certain sum over diagrams.  

  The Hilbert scheme series corresponding to 
  the summand of the HOMFLY homology with minimal ``\(a\)'' grading 
  can be recovered from the perverse filtration on
  the cohomology of the compactified Jacobian.  In the case of 
  $(k,n)$ torus knots, this space furnishes the unique finite
  dimensional simple
  representation of the rational spherical Cherednik algebra
  with central character $k/n$.  Up to a conjectural identification
  of the perverse filtration with a previously introduced
  filtration, the work of Haiman and Gordon and
  Stafford gives formulas for the Hilbert scheme series when 
    $k = mn + 1$. 
\end{abstract}

\maketitle

\section{Overview}

Let $X$ be the germ of a complex 
plane curve singularity.  Its topological
properties are captured by its {\em link}, the intersection of a 
representative of $X$ 
with the boundary of a small ball surrounding the singularity 
\cite{AGV,Mil}.

We previously conjectured \cite{OS} 
that the HOMFLY polynomial of the link is recovered from the 
Euler characteristics of certain moduli spaces associated to the singularity.
Specifically, let the HOMFLY polynomial 
$\overline{\Pp}$ be normalized by the following
skein relation: 
\begin{eqnarray}
  \label{eq:skein1}
  a \, \overline{\Pp}(\undercrossing) -  a^{-1} \, 
  \overline{\Pp}(\overcrossing) & = & (q - q^{-1})  
  \, \overline{\Pp}(\smoothing) \\
  a - a^{-1} & = & (q-q^{-1})\overline{\Pp}(\bigcirc)
\end{eqnarray}

We write $X^{[n]}$ for the Hilbert scheme of $n$ points on $X$. 
We define an incidence variety:
\[X^{[l]} \times X^{[l+m]} \supset 
X^{[l\le l + m]} := \{(I,J)\, | \, J  
\supset I  \supset M \cdot J \}\]
where $M$ is the maximal ideal at the central point.

\begin{conjecture} (\cite{OS})
  \label{conj:homfly}
  Let $X$ be the germ of a plane curve singularity, 
  with Milnor number $\mu$. Then
  \begin{equation*}
    (a/q)^{\mu-1} \sum_{l,m}
    q^{2l} (-a^2)^m \chi(X^{[l \le l + m]}) =     
        \overline{\Pp}(\mbox{Link of $X$}).
  \end{equation*}
\end{conjecture}

The object of the present article is to promote this to a homological 
conjecture.  On the right hand side, we replace the HOMFLY polynomial
with the Poincar\'e polynomial of
the triply graded HOMFLY homology of Khovanov and Rozansky \cite{KR}.  
This has several slight variants;
we discuss what is called in \cite{R} the
unreduced homology, and denoted by $\overline{\mathscr{H}}^{i,j,k}(K)$.  It is 
infinite dimensional, though finite in each graded piece.  
We write its graded dimension as
\[\uspp = \sum_{i,j,k} a^i q^j t^k \overline{\mathscr{H}}^{i,j,k}(K).\]
We discuss our grading conventions for \(\overline{\mathscr{H}}\) at the end of this section; for the moment, let us say that the homological grading \(t\) is chosen as in \cite{DGR}, rather than in \cite{KR}, so that  $\overline{\Pp} = \uspp|_{t=-1}$.

Recall that the cohomology of a complex algebraic variety admits a weight
filtration $W$, in terms of which one may form the 
{\em weight polynomial}:\footnote{Other authors
prefer the terms Serre polynomial, virtual Poincar\'e polynomial,
and E-polynomial.  Its existence was conjectured by Serre, and
follows from Deligne's theory of weights and mixed Hodge structures
\cite{Del-MHS, Del-weights}, for some discussion see
\cite{DK,Dur}.}
\[ \mathfrak{w}(X) = \sum_{j,k} (-1)^{j+k} t^k \mathrm{Gr}_W^k 
(\mathrm{H}^j_c(X)). \] 
The weight polynomial is characterized by two properties; first,
it agrees with the Poincar\'e polynomial if $X$ is a proper smooth variety, 
and second, it factors through the Grothendieck ring of varieties.
That is, for $Y$ closed in $X$, we have 
\[\wtp(X - Y) = \wtp(X) - \wtp(Y).\]

\begin{conjecture}
  \label{conj:homflyhomology}
  Let $X$ be the germ of a plane curve singularity, 
  with Milnor number $\mu$. Then
  \begin{equation*}
    (a/q)^{\mu-1} \sum_{l,m}
    q^{2l} a^{2m} t^{m^2} \mathfrak{w} (X^{[l \le l + m]}) =     
        \uspp(\mbox{Link of $X$}).
  \end{equation*}
\end{conjecture}

Throughout we write $\usppa$
for the LHS of Conjecture \ref{conj:homflyhomology}. 
As in \cite{OS} there is a useful equivalent formulation obtained 
by pushforward along the forgetful map $X^{[l\le l+m]} \to
X^{[l]}$.  By
Nakayama's lemma, the fibres are $\mathrm{Grass}(\CC^m \subset \CC^r)$ over 
the locus $X^{[l]}_r \subset X^{[l]}$
parameterizing subschemes whose ideals require $r$ generators. 
The weight polynomial of this Grassmannian is given by  
the q-binomial coefficient ${r \choose m}_{t^2}$.  Thus
by the identity
\[\sum_{m=0}^r {r \choose m}_{t^2}  a^{2m} t^{m^2} = \,\, 
  \prod_{k=1}^r (1+t^{2k-1} a^2)
\]
we may rewrite
\[\sum_{l,m} q^{2l} a^{2m} t^{m^2} \wtp(X^{[l\le l+m]}) = 
\sum_{l,r} q^{2l} \wtp(X^{[l]}_r) \prod_{k=1}^r (1+t^{2k-1} a^2) .\]

The series $\usppa$ enjoys the following symmetry and rationality properties:

\begin{proposition} \label{prop:rands} 
Let $X$ have $b$ branches and contribute
$\delta$ to the arithmetic genus.  Then the expression
$(q^{-1}-q)^b \usppa$ is a Laurent polynomial in
$q$ with terms of degrees  between $-2\delta$ and $2\delta$.  Moreover,
it is invariant under $q \to 1/qt$.  The number of different powers
of $a^2$ appearing is at most one more than the multiplicity of the singularity.
\end{proposition}

The polynomiality is known to hold for $\uspp$ -- for unibranch singularities the scaling factor
corresponds to taking reduced HOMFLY homology -- and
the invariance was predicted in \cite{DGR},
but remains conjectural. 
In the specialization of $\uspp$ 
to the HOMFLY polynomial, the symmetry $q \to -1/q$ is manifest in the
skein relation.  
 The bound on the degrees of $a$ which appear
corresponds to the fact that a singularity admits a braid presentation 
in which the number of strands is equal to the multiplicity.

\vspace{2mm}

It may be viewed as a defect that $\usppa$ is assembled from the
cohomologies of many different spaces.  In fact, 
 the coefficient of the lowest degree power of $a$,
\[\usppa^{\min{}}:= q^{1-\mu} \sum_{\ell = 0}^{\infty}
q^{2\ell} \wtp(X^{[\ell]})\] 
may be recovered 
from a single space, namely the compactified Jacobian.  As the previous 
proposition suggests, it is convenient to consider 
$$\sppa^{\min{}}:=(q^{-1} - q)^b \usppa^{\min{}}$$ where $b$ is the number
of analytic local branches of the singularity.  
If $C$ is a rational curve with $X$ as its unique
singularity, then we have
\[\sppa^{\min{}} = q^{-2g} (1-q^2)(1-q^2t^2) 
\sum_{\ell = 0}^\infty q^{2\ell} \wtp (C^{[\ell]}).\]

Let $\overline{J}^\ell(C)$ be the moduli space of rank one, 
degree $\ell$, 
torsion free sheaves on $C$; it is integral of 
dimension $g$ and locally a complete intersection \cite{AK, AIK}.
The choice
of a degree $\ell$ line bundle identifies $\overline{J}^\ell(C)$ with $\overline{J}^0(C)$.
We henceforth suppress the index for 
$\overline{J}(C) := \overline{J}^0(C)$, which
we term the {\it compactified Jacobian} of $C$. 
There is a map $C^{[\ell]}
\to \overline{J}^\ell(C)$ given by sending a subscheme to the dual of
the ideal sheaf cutting it out.   For $\ell \gg 0$ this map 
is a projective space bundle and so 
\(\mathfrak{w}(C^{[\ell]})\) is determined by \(\mathfrak{w}(\overline{J}(C))\).  In fact, work of  Maulik and Yun \cite{MY}, or of Migliorini and the third author \cite{MS}, shows that we can recover the entire series \(\sppa^{\min}\) from the
perverse filtration on the cohomology of \(\overline{J}_0(C)\). 
Specifically, according to \cite{FGvS} there is a
deformation 
$\pi:\cC \to B$ 
such that the total space $\overline{\jJ}$ of
the relative compactified Jacobian $\overline{\jJ} \to B$
is nonsingular. 
According to \cite{BBD} there is a 
decomposition $\mathrm{R} \pi_* \CC_{\overline{\jJ}}[g+\dim B] = 
\bigoplus \fF^i[-i]$ where the $\fF_i$ are perverse
sheaves and $i = -g, \ldots, g$.  
We write ${}^p \mathrm{H}^i(\overline{\jJ}_b):=
\fF^i|_{b}[-\dim B]$; this is a complex of vector spaces
carrying a weight filtration.  It can be shown
that this does not depend on the family $\cC$,
and that moreover: 

\begin{proposition}[\cite{MY,MS}]\label{prop:perverse}
$\displaystyle \sppa^{\min} = \sum\limits_{i=-g}^{g} q^{2i} \mathfrak{w}({}^p \mathrm{H}^i (\overline{J}^0(C))$
\end{proposition}

Suppose \(X\) is a unibranch singularity, and let \(K\) be its link. Let \( {\mathscr{H}}^{\min}(K)\) denote the part of the {\it reduced} HOMFLY homology with minimal \(a\)-grading (\(=\mu\)). Then in combination with Proposition~\ref{prop:perverse}, 
Conjecture~\ref{conj:homflyhomology} would imply that 
$${\mathscr{H}}^{\min}(K) \cong \mathrm{H}^*(\overline{J}^0(X)).$$
The homological grading on the LHS is identified  with the weight grading on the right, and the \(q\)-grading on the left with the perverse filtration on the right. 

\vskip2mm

To calculate $\usppa$, it is necessary to work out the 
deformation theory of ideals and nested ideals inside the local
ring of functions for the singularity in question.  In the case
when the singularity is described by a single Puiseux pair, {\it e.g.}  
if it is of the form $y^k = x^n$, we can reduce the calculation to
(nontrivial)  combinatorics. The argument 
  is similar  to Piontkowski's  calculation \cite{Pi} of the stratification
  of the compactified Jacobian of a rational curve with this singularity.

 Let 
$\Gamma_{k,n} = \{ak + bn\,|\,a,b \in \ZZ_{\ge 0} \}$ be the semigroup
of degrees of elements of $\CC[[t^n,t^k]]$.  We say
$\mathfrak{i} \subset \Gamma_{k,n}$ is a semigroup ideal if it is closed
under adding elements
from $\Gamma_{k,n}$.  

\begin{theorem}\label{thm:dimensions}
  Let $X_{k,n}$ be the germ of a singularity whose link
  is a $(k,n)$ torus knot, and let $\Gamma = \Gamma_{k,n}$. 
  Then $X_{k,n}^{[l \le l + m]}$ is stratified by linear spaces 
  enumerated by nested pairs of semigroup ideals
  $\mathfrak{j} \supset \mathfrak{i} \supset \mathfrak{j} +
  \{ak + bn\,|\,a,b \in \ZZ_{>0}\}$ such that
  $\# (\Gamma \setminus \mathfrak{j}) = l$ and $\# 
  (\mathfrak{j} \setminus \mathfrak{i}) = m$.  We write
  $N(\mathfrak{j} \supset \mathfrak{i})$ for the dimension
  of this linear space. 

  Fix $\mathfrak{j} \supset \mathfrak{i}$ and let
  $\{\gamma_1, \ldots, \gamma_r\}$ be the unique minimal
  subset of $\mathfrak{j}$ which generates it as a
  $\Gamma$ module.  
  Let $\sigma \mathfrak{j}$ be the set of all
  elements of $\mathfrak{j}$ with more than one expression
  of the form $j + \gamma$ with $j \in \mathfrak{j}, \gamma \in \Gamma$.
  Then $\sigma \mathfrak{j}$ is again a semigroup ideal requiring 
  $r$ generators, say $s_1,\ldots,s_r$, and
  \[ 
  N(\mathfrak{i} \supset \mathfrak{j}) =
  \sum_{\gamma_i \in \ii} \#(\Gamma_{> \gamma_i} \setminus \mathfrak{i})
  + 
  \sum_{\gamma_i \notin \ii} \#(\Gamma_{> \gamma_i} \setminus \mathfrak{j})
  - 
  \sum_{a=1}^r \#(\Gamma_{> s_a} \setminus \mathfrak{i}).
  \]
\end{theorem}

\begin{remark} 
 The existence of this stratification implies
 that there
  are no cancellations among monomials of $\usppa(X_{k,n})$ when setting
  $t=-1$.  It is an interesting question whether the analogous statement holds for all algebraic knots. 
\end{remark}

Using the formulas above, 
$\usppa$ can in principle be computed by
summing up the contributions; the computation is finite because
for any $\mathfrak{i} \subset c + \NN \subset \Gamma_{m,n}$, we have
$N(\mathfrak{i}) = N(\mathfrak{i}+1)$ and similarly for the nested case.  
In a certain limit, the formulas simplify: 

\begin{proposition}\label{prop:infty}
Write $X_{n,k}$ for the germ at the origin of
$y^n = x^k$ and  $X_{n,\infty}$ for that of  $y^n = 0$.  Then
\[ 
\sum_{\ell,m}
q^{2\ell} a^{2m} t^{m^2} \wtp(X_{n,k}^{[\ell \le \ell + m]})
=
\sum_{\ell,m}
q^{2\ell} a^{2m} t^{m^2} \wtp(X_{n,\infty}^{[\ell \le \ell + m]})
\,\,\, + \,\,\, O(q^{2k})\]
and we calculate
\[ 
\sum_{\ell,m}
q^{2\ell} a^{2m} t^{m^2} \wtp(X_{n,\infty}^{[\ell \le \ell + m]})
= \prod_{i=1}^{n} \frac{1+a^2 q^{2i-2} t^{2i-1}}
{1-q^{2i}t^{2i-2}}.
\]
\end{proposition}

This computation matches the formula for the ``stable superpolynomial'' of torus knots conjectured in \cite{DGR}.

To understand \(\usppa(X_{k,n})\) in general, it is profitable to consider the reformulation in terms of the compactified Jacobian. 
For the unibranch singularities $x^k = y^n$, the \(K\)-theory
of the compactified Jacobian is known \cite{VV} to furnish
a representation of the spherical 
Cherednik algebra of rank $n$ and central charge $k/n$.  
It will be shown elsewhere that the 
{\it rational} spherical Cherednik algebra acts on its cohomology \cite{OY}. 
Moreover, the homological grading (\(t\))  and the perverse 
filtration (\(q\)) have representation theoretic meanings. 

In case $k = mn + 1$ the graded dimensions of conjecturally equivalent filtrations
can be readily calculated in {a different} geometric incarnation of
the Cherednik algebra \cite{haiman2, GS,GS2}.  We obtain a formula
expressing $\sppa^{\mathrm{min}}(X_{mn+1,n})$ as a sum over partitions of $n$.
For a partition
$\lambda \vdash n$, and a box $x$ in the diagram of $\lambda$, 
we write $a(x), l(x)$ for the arm and leg, and $a'(x), l'(x)$ for the co-arm and 
co-leg.  We write $\lambda'$ for the dual partition, and 
$\kappa(\lambda)=\sum {\lambda'_i\choose 2}$.   We have 
the following formula: 

\begin{conjecture} \label{conj:mnp1min} 
Let $T_1 = q^2$ and $T^2 = 1/q^2 t^2$.  Then $ t^{-\mu} \sppa^{\mathrm{min}}(X_{mn+1,n})  $
is given by the following formula:
$$
\sum_{\lambda \vdash n} T_1^{m\kappa(\lambda)} T_2^{m\kappa(\lambda')} 
\frac{(1-T_1)(1-T_2) \prod_{x\in\lambda\setminus \{(0,0)\}}
  (1-T_1^{l'(x)}T_2^{a'(x)})}{\prod_{x\in\lambda }
  (1-T_1^{1+l(x)}T_2^{-a(x)})(1-T_1^{-l(x)}T_2^{1+a(x)})} \sum_{x \in \lambda} T_1^{l'(x)} T_2^{a'(x)}.
$$
\end{conjecture}

In a subsequent article
\cite{GORS},
we suggest how all of $\uspp(X_{k,n})$ (rather than
just $\uspp^{\min{}}$) may be recovered from the analogous
representation of the rational Cherednik algebra  (rather than just
the spherical part).  This leads to the following conjecture:

\begin{conjecture}\label{conj:mnp1} 
  Let $T_1 = q^2$, $T^2 = 1/q^2 t^2$, and $A = a^2 t$.  Then $(at)^{-\mu} \sppa(X_{mn+1,n})$
  is given by the following formula: 
$$ 
\sum_{\lambda \vdash n} T_1^{m\kappa(\lambda)} T_2^{m\kappa(\lambda')} 
\frac{(1-T_1)(1-T_2) \prod_{x\in\lambda\setminus \{(0,0)\}}
  (1-T_1^{l'(x)}T_2^{a'(x)})(1+A T_1^{-l'(x)}T_2^{-a'(x)})}
  {\prod_{x\in\lambda }
  (1-T_1^{1+l(x)}T_2^{-a(x)})(1-T_1^{-l(x)}T_2^{1+a(x)})} \sum_{x \in \lambda} T_1^{l'(x)} T_2^{a'(x)}.$$
\end{conjecture}

Formulas of this sort were first conjectured by Gorsky \cite{g} in the case of \(T(n,n+1)\), and subsequently in the physics literature by Dunin-Barkovsky  {\it et. al.} \cite{mm},
and Aganagic and Shakirov \cite{as} for \(T(n,nm+1)\).

\subsection*{Grading conventions for \(\overline{\mathscr{H}}\)}
Our normalization of the HOMFLY homology is the one used in \cite{DGR}, rather than that of \cite{KR} or \cite{R}. Specifically, our main interest is in the group \(\overline{\mathscr{H}}(K)\), which in the terminology of \cite{R} is the {\it  unreduced} HOMFLY homology. It satisfies 
 $$\overline{\mathscr{H}}(K) \cong \mathscr{H}(K) \otimes 
\mathrm{H}^*(S^1) \otimes \QQ[X]$$
 where the group 
 \( \mathscr{H}(K)\) is the {\it reduced} HOMFLY homology. 
 
 We normalize the homological (\(t\)) grading on \( \mathscr{H}(K)\) so as to coincide with the homological grading on reduced Khovanov homology under the spectral sequence of \cite{R}; for example, the Poincar{\'e} polynomial of  \( \mathscr{H}\) of the positive trefoil is given by
 $$a^2q^{-2}t^0+a^4q^0t^3 + a^2q^2t^2.$$
 
 The Poincar{\'e} polynomials of the reduced and unreduced homologies are related by:
 $$\overline{\mathscr{P}} (K) =\frac{at+a^{-1}}{q-q^{-1}} \mathscr{P}(K).$$
 
 The Poincar{\'e} polynomial with respect to the homological grading (\(s\)) of \cite{KR} may be obtained by substituting  \(t=s^{-1}, a^2 =a^2s\) in \(\overline{\mathscr{P}}\). This reflects the fact that the 
 homological grading on \(sl(n)\) homology is obtained as a linear combination of the \(a\) and \(s\) gradings \cite{R}, and the fact that  \(sl(2)\) homology is dual to Khovanov homology. 
 
\vspace{2mm} \noindent {\bf Acknowledgements.}  
We thank Emanuel Diaconescu, Pavel Etingof, Lothar G\"ottsche, Ian Grojnowski, Mikhail Khovanov, Ivan Losev, Luca
Migliorini, Rahul Pandharipande, Ivan Smith, Cumrun Vafa, Ben Webster, Geordie Williamson, 
and Zhiwei Yun for helpful
discussions. A. Oblomkov is supported by Sloan Foundation and NSF.
J. Rasmussen would like to thank the Simons Center for Geometry and Physics for its hospitality and support while much of this work was done. 
V. Shende is supported by 
the Simons Foundation.

\section{Hilbert Schemes and Jacobians} \label{sec:BPS}

In this section, we develop some general facts about the series \(\usppa\):
its rationality and symmetry properties, behavior under blowups, and  relation to the cohomology of the compactified Jacobian. For the most part, these are straightforward consequences of previous work.

\subsection{Rationality and Symmetry}
In this section, we prove  Proposition~\ref{prop:rands}. We begin by reviewing
\cite[Section 4]{OS} and \cite[Appendix B]{PT3}, presenting them
now in the Grothendieck ring of varieties.  This is the
ring generated by the classes $[V]$ of varieties $V$, the sum
and product coming from disjoint union and direct product
respectively.  The classes are subject
to the relation
$[V \setminus Z] = [V] - [Z]$ for $Z$ is a closed subvariety of $V$.  
If $\phi:V \to A$ is a constructible function, we write 
$[V,\phi]: = \sum_{a \in A} a\cdot [\phi^{-1}(a)]$.  

It is convenient to rewrite slightly the 
series $\uspp_{\mathrm{alg}}$. 
Consider the locus $X^{[l]}_r \subset X^{[l]}$
parameterizing subschemes whose ideals require $r$ generators.  By
Nakayama's lemma, the projection $X^{[l\le l+m]} \to X^{[l]}$
has fibres $\mathrm{Grass}(\CC^r \subset \CC^m)$ over $X^{[l]}_r$.
Thus we have
\[\sum_{l,m} q^{2l} a^{2m} t^{m^2} [X^{[l\le l+m]}] = 
\sum_{l,r} q^{2l} [X^{[l]}_r] \sum_m a^{2m} t^{m^2} [G(r,m)] = 
\sum_{l} q^{2l} [X^{[l]}, \Phi] \]
where the constructible function 
$\Phi$ on $X^{[l]}$ takes the value 
$\sum_m a^{2m} t^{m^2} [G(r,m)]$ at a point corresponding
to an ideal requiring $m$ generators.  

We pass to a complete curve $C$ with a unique singularity at $p$, at
which its germ is $X$.
There is a stratified
map $C^{[l]} \to \coprod_{l' \le l} X^{[l']}$ which forgets
points away from the singularity; we extend $\Phi$ to $C^{[l]}$ by pullback
along this map.  Then by a standard stratification argument we have
\[  \sum_{l} q^{2l} [X^{[l]}, \Phi] = 
\frac{\sum_{l} q^{2l} [C^{[l]}, \Phi]}
{(1-q^2)^b \sum_{l} q^{2l} [\widetilde{C}^{[l]}]} \] 
where $\widetilde{C}$ is the normalization of $C$.

Essentially by definition, $\Phi$ depends only on the
isomorphism class as a sheaf of the ideal sheaf of the subscheme, and 
does not change when this sheaf is tensored by a line bundle.  Moreover,
it is shown in \cite{OS} (using the planarity of the singularities
and applying the Auslander-Buchsbaum theorem) that $\Phi$ of a sheaf
and its dual agree.  In other words, it satisfies the hypothesis
of the following lemma, whose statement and proof are modeled on Lemma 3.13 of \cite{PT3}.

\begin{lemma}  \label{lem:bps}
Let $C$ be a Gorenstein curve of arithmetic genus $g$.
Let $\phi$ be a constructible function on the moduli space
$\overline{J}(C)$ of rank one, torsion free sheaves.
Assume that for any rank one torsion free sheaf $\fF$ and
any line bundle $L$, we have $\phi(\fF^*) = \phi(\fF) = \phi(\fF \otimes L)$. 
Denote also by $\phi$ the function induced on $C^{[n]}$ by composition
with the Abel-Jacobi map. 

Then there exist classes $N_h(C,\phi)$ in the Grothendieck group of varieties
(with coefficients in the ring in which $\phi$ takes values) 
such that
\[(1-q)(1-q\LL) \sum_{n=0}^\infty q^n [C^{[n]}, \phi] = 
\sum_{h=0}^g N_h(C,\phi) \cdot (1-q)^{h} (1-q\LL)^{h} q^{g-h},\]
where \(\LL = [\AA^1]\) is the class of the affine line. 
In particular, denoting by $Z_C(q)$ the quantity on either side of the above
formula, we have $Z_C(q) = (q^2 \LL )^g Z_C(1/q\LL)$. 
\end{lemma}

\begin{proof}
  The following useful properties
  of Gorenstein curves may be found in 
  an article of Hartshorne \cite{H}.
  Let $C$ be a Gorenstein curve, and let $F$ 
  be a torsion free sheaf on $C$. 
  Write $F^*$ for $\hH om (F,\oO_C)$. 
  Then $\mathcal{E}xt^{\ge 1}(F,\oO_C) = 0$ and 
  $F = (F^*)^*$.  Serre duality
  holds in the form $\mathrm{H}^i(F) = \mathrm{H}^{1-i}(F^* \otimes \omega_C)^*$.
  For $F$ rank one and torsion free, define its 
  degree $d(F):=\chi(F) - \chi(\oO_C)$.   This satisfies 
  $d(F) = - d(F^*)$, and, for $L$ any line bundle, 
  $d(F \otimes L) = d(F) + d(L)$. 

  We proceed with the proof of the lemma.  Fix a degree 1 line bundle 
  $\oO(1)$ on $C$.  We map $C^{[d]} \to \overline{J}^0(C)$
  by associating the ideal $I \subset \oO_C$ to the
  sheaf $I^* = \hH om(I,\oO_C) \otimes \oO(-d)$; the fibre is 
  $\PP(\mathrm{H}^0(C,I^*))$.  For $\fF$ a rank one degree zero torsion
  free sheaf, we write the Hilbert function as 
  $h_\fF(d) = \dim \mathrm{H}^0(C,\fF \otimes \oO(d))$.  Then since over
  the strata with constant Hilbert function, the map from the Hilbert schemes
  to the compactified Jacobian is the projectivization of a vector bundle, 
  we have the equality
  \[\sum_{d=0}^\infty q^d [C^{[d]},\phi]  = \sum_{a \in A} \sum_h a  
  [\{\fF\,|\,h_{\fF} = h, \, 
  \phi(\fF) = a\}] \sum_{d=0}^{\infty} q^d 
  [\PP^{h(d)-1}].\]
  
  Fix $h = h_\fF$ for some $\fF$.  Evidently $h$ 
  is supported in $[0,\infty)$, and by Riemann-Roch and
  Serre duality is equal to
  $d+1-g$ in $(2g-2,\infty)$.  Inside $[0,2g-2]$, it
   increases by either 0 or 1 at each step.  Let $\phi_\pm(h)
  = \{ d\,|\, 2 h(d-1) - h(d-2) - h(d) = \pm 1\}$; evidently
  $\phi_- \subset [0,2g]$ and $\phi_+ \subset [1,2g-1]$, and
  \[Z_h(q):=(1-q) (1-q\LL) \sum_{d=0}^\infty q^d [\PP^{h(d)-1}] = 
  \sum_{d \in \phi_-(\fF)} q^d \LL^{h(d)-1} -
  \sum_{d \in \phi_+(\fF)} 
  q^{d} \LL^{h(d-1)} \]
  This is a polynomial in $q$ of degree at most $2g$, hence so is $Z_C(q)$. 
  
  Let $\gG = \fF^* \otimes \omega_C \otimes \oO(2-2g)$, and $h^\vee = h_\gG$. 
  By Serre duality and Riemann-Roch, $h^\vee(d) = h(2g-2-d) + d + 1 - g$, so
  in particular, $d \in \phi_{\pm}(h^\vee) \iff 2g-d \in \phi_{\pm}(h)$.  It
  follows that $q^{2g} \LL^g  Z_h(1/q\LL) = Z_{h^{\vee}}(q)$.  As 
  $Z_C(q) = \sum_h  [\{\fF\,|\,h_{\fF} = h\}] Z_h(q)$, we obtain the 
  final stated equality. 
\end{proof}

Take the curve $C$ to be rational, recall that  
$\mu = 2\delta + 1 - b$ \cite{Mil}, and note that in Lemma \ref{lem:bps}
we used the variable $q$ rather than $q^2$.  Then we see that
there exist classes  $N_h$ in the Grothendieck ring of varieties
(with coefficients in $\ZZ[a,t]$) such that 
\[ \uspp_{\mathrm{alg}} = \mathfrak{w} \left( a^{2\delta - b} 
(q^{-1}-q)^{-b} \sum_{h=0}^\delta N_h(C,\phi) 
\cdot (q^{-1}-q)^{h} (q^{-1}-q\LL)^{h} \right) \]
From this expression, we see immediately that 
$(q^{-1} -q)^b \uspp_{\mathrm{alg}}$ is a Laurent polynomial 
in $q$ with coefficients between $q^{-2\delta}$ and $q^{2 \delta}$ which
is invariant under $q \to 1/qt$.  Finally note the degree of
$N_h$ in $a^2$ is bounded by the multiplicity of the singularity, 
as this is the maximal number of generators which any ideal 
will require  \cite[Exercise 4.6.16]{BH}.  This completes
the proof of Proposition \ref{prop:rands}. \qed
\vskip2mm
\noindent{\it Remark:} 
When \(b=1\), the link of \(X\) is a knot, and the product \((q-q^{-1})\usppa(X)\) corresponds under Conjecture~\ref{conj:homflyhomology} to \(a^{-1}+at\) times the Poincar{\'e} polynomial of \(\mathscr{H}(K)\). In particular, all of its terms are positive. For \(b>1\), the quantity 
\((q-q^{-1})^b\usppa(X)\) is somewhat less natural from the point of view of the HOMFLY homology. For a two-component link,  \((q-q^{-1})^2\usppa(X)\) will typically have negative terms, so it cannot coincide with the Poincar{\'e} polynomial of the completely reduced homology considered in \cite{R}. 
\subsection{Blowups}
If \(X\) is the germ of a plane curve singularity, we can blow up \(X\) at the central point to obtain  the germ of a new singularity \(\widetilde{X}\). The effect of this operation on the link of the singularity is well-known. If \(m\) is the multiplicity of \(X\), we can write the link of \(X\) as the closure of a \(m\)-strand braid \(\sigma\). Then the link of \(\widetilde{X}\) is the closure of the braid \(\sigma \Delta^{-2}\), where \(\Delta^{-2}\) denotes a full left-handed twist on \(m\) strands. 

Now let \(\sigma\) be an arbitrary braid with closure \(\overline{\sigma}\), and let \(n\) and \(N\) be the minimum and maximum powers of \(a\) appearing in \(\overline{\Pp}(\overline{\sigma})\). The Morton-Franks-Williams inequality says that 
$  w-m \leq n \leq N \leq w+m $, where \(w\) and \(m\) denote the writhe and number of braid strands in \(m\) respectively. Let \(\overline{\Pp}^{\rm{min}}\) and \( \overline{\Pp}^{\rm{max}}\) be the coefficients of \(a^{w-m}\) and \(a^{w+m}\) in \(\overline{\Pp}(\overline{\sigma})\). Then according to a theorem of  Kalman \cite{Kalman},
$\overline{\Pp}^{\rm{min}}(\overline{\sigma\Delta^{-2}}) = \overline{\Pp}^{\rm{max}}(\overline{\sigma}).$
For algebraic knots, an analogous statement holds at the level of \(\usppa\):
\begin{proposition}
If \(X\) is the germ of a unibranch plane curve singularity with multiplicity \(m\), then
\(\usppa^{\rm{max}}(X) = t^{m^2} \usppa^{\rm{min}}(\widetilde{X}) \).
\end{proposition}

\begin{proof}
It is a classical result that \(\mu(X) = \mu(\widetilde{X})+m(m-1)\). 
Thus the proposition  is equivalent to saying that 
$$\sum_{l=0}^\infty q^{2l} \mathfrak{w}(X^{[l\leq l+m]}) = 
q^{m(m-1)} \sum_{l=0}^\infty q^{2l}\mathfrak{w}(\widetilde{X}^{[l]}).$$
We will show that  \(X^{[l+m(m-1)/2\leq l+m(m+1)/2]} \cong \widetilde{X}^{[l]}\). 

Let us write $\oO_X = \CC[[x,y]]/f(x,y)$. 
By the Weierstrass preparation theorem and unibrachness
we may choose variables so that
\[f(x,y) = y^m + 
y^{m-1} x^2 f_1(x) + y^{m-2} x^3 f_2(x) + \ldots + x^{m+1} f_m(x)\]
Then $x^{-m} f(x, x. \frac{y}{x}) \in \CC[[x,\frac{y}{x}]]$, and 
$\oO_{\widetilde{X}} = \CC[[x,\frac{y}{x}]]/x^{-m} f(x,x.\frac{y}{x})$.  In particular
$\oO_{\widetilde{X}} / \oO_X$ is generated as an $\oO_X$-module by 
\(1,\frac{y}{x},\ldots,(\frac{y}{x})^{m-1}\).  Therefore, 
\[x^{m-1} \oO_{\widetilde{X}} = (x^{m-1}, x^{m-2} y, \ldots, y^{m-1}) 
\subset \oO_X\]
As $f \in (x,y)^{m-1}$, we have 
\[\dim \oO_X / (x,y)^{m-1} = \dim \CC[[x,y]]/(x,y)^{m-1} = m(m-1)/2.\] 
Moreover for any $J \subset \oO_{\widetilde{X}}$, we have 
\[\dim \oO_X / x^{m-1} J = \dim \oO_X/ x^{m-1} \oO_{\widetilde{X}} +
\dim \oO_{\widetilde{X}} / J = m(m-1)/2 + \dim \oO_{\widetilde{X}}/J.\]
An identical argument shows 
$\dim \oO_X / x^m J = m(m+1)/2 + \dim \oO_{\widetilde{X}}/J$.  Finally 
$x^m J \subset (x,y) x^{m-1} J$.  In fact, these 
are equal since $\frac{y}{x} J \subset J$.

Therefore we define a map 
\(\Phi: \widetilde{X}^{[l]} \to X^{[l+m(m-1)/2\leq l+m(m+1)/2]}\) by 
\(\Phi(J) = (x^{m-1} J,x^m J)\).  The map is injective
because $x$ is not a zero divisor. 
 
To see that \(\Phi\) is surjective, suppose we are given $\oO_X$-ideals 
\(\oO_X \supset J \supset I \supset (x,y) J\) with \(\dim (J/I)  = m\).  
As $J$ is a {\em free} $\CC[[x]]$-module of rank $m$,
$\dim_\CC J/xJ = m$ and therefore $I = xJ$.  On the other hand since
$X$ has multiplicity $m$, we certainly have $\dim_\CC J/(x,y)J \le m$.
We conclude $xJ = (x,y)J$.  In particular $yJ \subset xJ$ therefore
$\frac{y}{x}J \subset J$, i.e., $J$ is an $\oO_{\widetilde{X}}$-module. 
It is elementary to show that {\em any} $\oO_{\widetilde{X}}$-submodule
$J \subset \oO_X$ must satisfy $J \subset x^{m-1} \oO_{\widetilde{X}}$. 
\end{proof}

\begin{remark} 
 The result holds for  non-unibranch singularities as well; the proof is more technical and will appear elsewhere.
\end{remark}

\subsection{Relation with the compactified Jacobian}
In this subsection, we provide  background and context for the statement of
 Proposition~\ref{prop:perverse}. 
Recall that 
for a {\em smooth} curve $C$, the Hilbert schemes $C^{[n]}$ are
just symmetric products, and as such their cohomology
may be computed by taking ${S}_n$ invariants: 
$\mathrm{H}^*(C^{[n]},\CC) = \mathrm{H}^*(C^n,\CC)^{{S}_n}
= (\mathrm{H}^*(C,\CC)^{\otimes n})^{{S}_n}$.  
On the other hand, $\mathrm{H}^i(J(C),\CC) = \bigwedge^i 
\mathrm{H}^1(C,\CC)$.
This leads to the following formula of Macdonald \cite{Mac}: 
\begin{equation} \label{eq:macdonaldseries}
  \sum_{d=0}^\infty \sum_{i=0}^{2d} q^{2d} \mathrm{H}^i (C^{[d]}, \CC)
  = \frac{ \sum\limits_{i=0}^{2g} 
    q^{2i} \bigwedge ^i (\mathrm{H}^1 (C, \CC)) }
 {(1-q^2)(1-q^2 \CC(-1) )}
  = 
  \frac{\sum\limits_{i=0}^{2g} q^{2i} \mathrm{H}^i(J(C),\CC)}{(1-q^2)
    (1-q^2 \CC(-1))} .
\end{equation}
The Tate twists in the denominator are necessary to make this an 
equality of Hodge structures.  Since all spaces are smooth and compact, 
taking weight polynomials amounts to replacing 
$\CC(-1)$ by $t^2$ and $\mathrm{H}^i(\cdots)$ by $t^i \dim \mathrm{H}^i(\cdots)$. 

Proposition~\ref{prop:perverse} says that an analogous formula holds for a {\it singular} plane curve \(C\), but we must take into account the  perverse filtration on \(\mathrm{H}^*(\overline{J}^0(C)) \). More precisely, 
let ${}^p  \mathrm{H}^i (\overline{J}^0(C))$ 
be the \(i\)th associated graded piece 
of the perverse filtration on \(\mathrm{H}^*(\overline{J}^0(C))\). Then
the main result of \cite{MY},\cite{MS} is that 

\begin{equation} 
  (1-q^2)(1-q^2 t^2) \sum_{d=0}^\infty q^{2d} \mathfrak{w} (C^{[d]}) 
  = 
  \sum\limits_{i=0}^{2g} q^{2i} \mathfrak{w}({}^p \mathrm{H}^i (\overline{J}^0(C)).
\end{equation}

The $q \to 1/qt$ symmetry of the LHS proven in Proposition 
\ref{prop:rands} manifests on the RHS as (relative) Poincar\'e duality. 

We  recall the definition of the perverse filtration for the interested reader. 
 Let $\pi: \cC \to B$ be a family of curves over a
smooth base, with the general fibre smooth and some special singular
fibre $C= \cC_b$ we are interested in.  Let $\pi^J:\mathcal{J} \to B$ 
be the relative compactified Jacobian.  As shown in \cite{FGvS}, 
there exist families such that $\mathcal{J}$ is smooth; fix any such. 
Then from the decomposition
theorem of Beilinson, Bernstein, and Deligne \cite{BBD} we learn that 
$$\mathrm{R}\pi^J_* \CC[g + \dim B] = \bigoplus \mathrm 
({}^p \mathrm{R}^{g+i} \pi^J_* \CC[g + \dim B])[-i].$$ 
 Passing to the central fibre, we
write $${}^p \mathrm{H}^i (\jJ_b) := ({}^p \mathrm{R}^{g+i} \pi^J_* \CC[g + \dim B])_b[-g-i].$$
These are complexes of mixed Hodge modules, which should be Tate-twisted to ensure 
$\mathrm{H}^j(\jJ_b) = \bigoplus \mathrm{H}^j ({}^p  \mathrm{H}^i (\jJ_b))$.\footnote{This
direct sum decomposition is not canonical, but it does come from a canonical filtration;
these matters are irrelevant here.}

It is sensible to take weight polynomials of the ${}^p \mathrm{H}^i (\jJ_b)$.
It is shown in \cite{MY, MS} that
\begin{equation} \label{eq:myms}
  (1-q^2)(1-q^2 t^2) \sum_{d=0}^\infty \sum_{i=0}^{2d} q^{2d} \mathfrak{w} (\cC_b^{[d]}) 
  = 
  \sum\limits_{i=-g}^{g} q^{2i} \mathfrak{w}({}^p \mathrm{H}^i (\jJ_b)).
\end{equation}
As the LHS did not depend on the family, we learn a posteriori 
the same for the RHS.
Equation \ref{eq:myms} is proven by 
showing that no summand of 
$\mathrm{R}\pi^J_* \CC[g + \dim B]$ and $\mathrm{R}\pi^{[d]}_* \CC[d + \dim B]$
is supported in positive codimension, and thus we can check on the generic point
where the assertion reduces to Equation (\ref{eq:macdonaldseries}).  
For the Jacobians
this follows from Ngo's support theorem \cite{N}; 
the geometric content of \cite{MY, MS} was to
establish the same for the Hilbert schemes.

\section{Equations for  Hilbert schemes} \label{sec:syz}

In this section we prove that the Hilbert schemes and nested Hilbert schemes
for unibranch singularities with a single Puiseux pair (e.g. $x^k = y^n$ with
$k, n$ coprime) admit a stratification by cells which admit bijective morphisms
from various $\AA^N$.  We give explicitly the dimensions $N$ in terms of certain
combinatorial data. 
 Our approach is adapted from the methods of Piontkowski \cite{Pi}.

\subsection{Semigroups, stratifications, and syzygies}
 Let $X$ be the germ of a unibranch plane curve singularity
 with complete local ring $\oO_X \subset\mathbb{C}[[t]].$
 We have the valuation $\ord: \CC[[t]]\setminus\{0\} \to \NN$ 
 which takes $\lambda t^k + O(t^{k+1})  \mapsto  k$.
 The set  $\Gamma:=\ord(\mathcal{O}_X\setminus \{0\})$ 
 is a semigroup.
 Given an ideal $J\subset \mathcal{O}_X$ the set $\ord(J\setminus\{0\})$ 
 is a $\Gamma$-module which we call the {\it symbol} of $J$.   
 We will study the geometry of the moduli
 space of ideals with given symbol:
 $$X^{[\jj]}:=\{ J\subset \mathcal{O}_X|\ord(J)=\jj\}.$$

 Let $\CC[\Gamma] = \CC[t^i | i \in \Gamma ]$.  
 For a $\Gamma$-submodule $\jj \subset \ZZ$, we write
 $\underline{\jj} = (t^j|j\in \jj) \CC[\Gamma]$ for the associated monomial
 ideal, and $c(\jj):=\min\{j \in \jj\,|\,j+\NN \subset \jj\}$ for its conductor.

 We choose a basis of $\oO_X$ compatible with the monomial basis of 
 $\CC[\Gamma]$:
 $$\phi_i=t^i+\sum_{j>i} a_{ij} t^j,\quad i\in \Gamma.$$ 
 Evidently the multiplication matrix in this basis is upper triangular:
 \begin{equation} \label{eq:uppert}
   \phi_\alpha\cdot\phi_\beta= \phi_{\alpha+\beta} + 
   \sum_{\gamma>\alpha+\beta} C^\gamma_{\alpha,\beta}\phi_\gamma
 \end{equation}

 Below we use the following map:
 $$ \gr: \mathcal{O}_X\to \CC[\Gamma]:\quad \gr(f):=t^{\ord(f)}([t^{\ord(f)}]f).$$
 where \(([t^i]f)\) denotes the coefficient of \(t^i\) in \(f\). 
 We say $f\in \oO_X$ lifts $at^i\in \CC[\Gamma]$ iff $\gr(f)=at^i$.

Suppose we are given an ideal $J$ with symbol $\jj$. 
For any element $j\in\jj$ there is a unique element $\tau_j\in J$ of the form:
$$\tau_j=\phi_j+\sum_{k\in \Gamma_{>j}\setminus\jj}\lambda^{k-j}\phi_k.$$
This observation motivates us to study the following map. 
Fix generators
$\gamma_1,\dots,\gamma_n$ of $\jj$.  Let $Gen = \Spec\,\CC[\lambda_j^{k-\gamma_j}\,|\,k 
\in \Gamma_{>\gamma_j}\setminus\jj]$ be an affine space of dimension 
$N=\sum_j |\Gamma_{>\gamma_j}\setminus\jj|$.  Then we define deformations of the generators
\[\tau_{\gamma_j}(\lambda_\bullet^\bullet)=
\phi_{\gamma_j}+\sum_{k\in\Gamma_{>\gamma_j}\setminus\jj}\lambda^{k-\gamma_j}_j\phi_k\]
and an ``exponential'' map 
\begin{eqnarray*}
  \Exp_\gamma: Gen & \to & \coprod X^{[n]} \\ 
    \lambda & \mapsto & (\tau_{\gamma_1},\dots,\tau_{\gamma_n})
\end{eqnarray*}

Note that $\Exp_\gamma(\lambda)$ may have different colengths at different 
$\lambda$; taking a flattening stratification shows
$\Exp_\gamma$ is constructible. It is easy to see from Theorem~27 in \cite{OS}  that
$\Exp_\gamma^{-1}(X^{[\# \Gamma \setminus \jj]}) = \Exp_\gamma^{-1}(X^{[\jj]})$ and
that the map restricts to a bijective morphism
$\Exp_\gamma: \Exp_\gamma^{-1}(X^{[\jj]})\to X^{[\jj]}$.  
From now on, we tacitly  identify \(X^{[\jj]}\) with \(  \Exp_\gamma^{-1}(X^{[\jj]}) \subset Gen\).
We illustrate the behavior at points $\lambda \in Gen \setminus \Exp_\gamma^{-1}(X^{[\jj]})$:

\begin{example} 
Let $\oO_X = \CC[[t^3,t^7]]$ and $\jj = \langle 6, 10 \rangle$.
Then $(t^6 + t^7, t^{10}) \in Gen$ 
is generated by elements of orders $6$ and $10$, and $14 \notin 
\langle 6, 10 \rangle$.  On the other hand 
$\ord (t^6 + t^7, t^{10}) = \langle 6, 10, 14 \rangle$ because
$t^7 (t^6 + t^7) - t^3(t^{10}) = t^{14}$, so $(t^6 + t^7, t^{10}) \notin 
\Exp_\gamma^{-1}(X^{[\jj]})$.
\end{example}

That is, although the orders of the generators of the ideal 
$\Exp(\lambda)$ generate $\jj$, it can and does happen that 
some $\oO_X$-linear combination of the generators has order 
$\notin \jj$.
To prevent this from happening one needs to control the 
syzygies (relations 
between generators).

The choice of generators $\gamma_i$ of the ideal $\jj$
determines a surjection $G: \CC[\Gamma]^{\oplus n} \to \underline{\jj}$.  
Extend this to a presentation: 
\begin{equation}\label{GammaExSeq}
  0\leftarrow\CC[\Gamma]/\underline{\jj}\leftarrow \CC[\Gamma] \xleftarrow{G}
  \CC[\Gamma]^{\oplus n}\xleftarrow{S}\CC[\Gamma]^{\oplus m}
\end{equation}
where the matrix of syzygies $S=(\vec{s}_1,\dots,\vec{s}_m)$
is homogenous, in the sense that 
$(\vec{s}_i)_j=u^j_it^{\sigma_i-\gamma_j}$ for some constants $u^j_i\in\CC$. 
We call $\sigma_i$ the order of the syzygy $\vec{s}_i$.
We regard $G$ as a row vector, and $\vec{s}_i$ as the columns of $S$. 

The choice of $\lambda\in Gen$ determines
a lift $\mathcal{G}_\lambda\in 
\Hom_{\oO_X}(\oO_X^{\oplus n},\oO_X)$ of $G$
by the formula
$$(\mathcal{G}_\lambda)_j:=\tau_{\gamma_j}(\lambda^\bullet_\bullet) = \phi_{\gamma_j}+
\sum_{k\in\Gamma_{>\gamma_j}\setminus\jj}\lambda^{k-\gamma_j}_j\phi_k.$$

We define $\ord$ and $\gr$ on $\mathcal{O}_X^{\oplus n}$, note that these
{\em do not} act entry-wise:
\begin{gather*}
\ord: \oO_X^{\oplus n}\to \ZZ: \quad \ord(v)=\min\{\ord(v_j)+\gamma_j\},\\
gr:\oO_X^{\oplus n}\to \CC[\Gamma]^n:\quad 
gr(v)_j:=t^{d-\gamma_j}[t^{d-\gamma_j}]v_j,\quad d=\ord(v).
\end{gather*}
We define 
$gr: \Hom_{\oO_X}(\oO_X^{\oplus m},\oO_X^{\oplus n})\to 
\Hom_{\CC[\Gamma]}(\CC[\Gamma]^{\oplus m},\CC[\Gamma]^{\oplus n})$ column by column
by the formula $\gr({s}_1,\ldots,{s}_m) = 
(\gr  {s}_1,\ldots,\gr {s}_m)$. 

\begin{lemma} \label{lem:cutbysyz}
Fix $\lambda \in Gen$. The following are equivalent
\begin{enumerate}
  \item There exists $\sS \in \Hom_{\oO_{X}}(\oO_{{X}}^{\oplus m},\oO_{{X}}^{\oplus n})$
    such that 
    $\gr(\mathcal{S})=S$ and every entry of
    $\mathcal{G}_\lambda\circ\mathcal{S}$ has order at least 
    $c(\jj)$.
  \item There exists 
    $\widetilde{\sS} \in \Hom_{\oO_{X}}(\oO_{{X}}^{\oplus m},\oO_{{X}}^{\oplus n})$
    such that 
    $gr(\widetilde{\mathcal{S}})=S$ and 
    $\mathcal{G}_\lambda\circ\widetilde{\mathcal{S}}=0$
  \item $\lambda \in X^{[\jj]}$. 
\end{enumerate}
\end{lemma}
\begin{proof}
  ($1 \implies 2$).  From a column $s$ 
  of $\sS$, we will produce a column
  $\widetilde{s}$ of $\widetilde{\sS}$.  By the hypothesis 
 each term of  of $\gG_\lambda s$ is in \(\jj\). Thus we can write  $\gG_\lambda s = \sum f_i \tau_{\gamma_i}$ for some $f_i$.
  Then set $\widetilde{s} = s - (f_1,\ldots,f_n)$.  The
  converse ($2 \implies 1$) is obvious.

  ($3 \implies 2$). 
  Let $\vec{s}$ be a column of $S$; then  
  $\vec{s}_j = c_j t^{l-\gamma_j}$ 
  for some constants $c_i$ such that $\sum c_j = 0$. 
  Let us define a first approximation $s^l$ by $(s^l)_j = 
  c_j \phi_{l-\gamma_i}$; evidently
  $\gr(s^l) = \vec{s}$ and each term of $\gG_\lambda s^l$ has order
  greater than $l$.  By the hypothesis these terms have order
  in  $\jj_{>l}$.   Assume now we have found $s^r$ such that
  $\gr(s^r) = \vec{s}$ and each term of $\gG_\lambda s^r$ has order
  in $\jj_{>r}$.  Then 
  we can write ${\rm gr}(\gG_\lambda s^r) =\sum_{j=1}^n b_j t^{d-\gamma_j} t^{\gamma_j}$,
  where $d={\rm ord}(\gG_\lambda s^r)$ and $b_j t^{d-\gamma_j}\in \CC[\Gamma]$, $j=1,\dots,n$.
 The induction step is given by the formula
  $s^{r+1} = s^r - (b_1 \phi_{d-\gamma_1},\ldots,b_n\phi_{d-\gamma_n})$.

  ($2 \implies 3$).  Assume $\Exp_\gamma(\lambda) \notin X^{[\jj]}$. 
  Then there exists $\varphi \in \oO_X^{\oplus n}$ such that
  $\ord(\mathcal{G}_\lambda(\varphi))\notin \jj$.  Observe
  this implies  $G(\gr(\varphi_1),\dots,\gr(\varphi_n))=0$.
  Therefore we may find $\psi\in \oO_X^{\oplus n}$ such that 
  $\mathcal{G}_\lambda(\psi)=0$ and $\gr(\psi)=\gr(\varphi)$.
  We have $\ord(\varphi-\psi)>\ord(\varphi)$ and
  $\mathcal{G}_\lambda(\varphi-\psi)=
  \mathcal{G}_\lambda(\varphi)$.  Continuing this process
  we may produce an element $\widetilde{\varphi}$ 
  of arbitrarily high order with $\ord(\mathcal{G}_\lambda(\varphi))\notin \jj$.
  However, $\ord(\mathcal{G}_\lambda (\widetilde{\varphi})) > \ord \widetilde{\varphi}$, 
  so once $\widetilde{\varphi} > c(\jj)$ we find a contradiction.
\end{proof}

Let $\widetilde{Syz}$ be the (infinite dimensional) space 
parameterizing possible
syzygies. 
From the lemma it follows that $X^{[\jj]}$ is the image in
$Gen$ of the variety in $\xX \subset Gen \times \widetilde{Syz}$ 
cut out by the equation
$\gG \circ \sS = 0$.  We now describe a 
finite dimensional affine slice $Syz \subset \widetilde{Syz}$ 
such that $\xX \cap Syz = X^{[\jj]}$. For each $s \in \jj_{\le c(\jj)}$ fix a decomposition 
$s = \gamma_{g(s)} + \rho(s)$ for $\rho(s) \in \Gamma$.  

Let $Syz$ be the affine space with 
coordinates $\nu_{is}^{s-\sigma_i}$ where 
$i=1,\dots,m$ and $c(\jj)>s>\sigma_i$.  To a point $\nu$ in $Syz$ we assign
a $n\times m$ matrix with entries:
$$(\mathcal{S}_\nu)_i^j=u_i^j\phi_{\sigma_i-\gamma_j}+
\sum_{\substack{c(\jj) > s >\sigma_i \\ g(s) = j}}
\nu_{is}^{s-\sigma_i}\phi_{s - \gamma_j}.$$

\begin{proposition}\label{Wjj}
The subvariety of  $Gen\times Syz$
defined by the equation 
$\mathcal{G}_\lambda\circ\mathcal{S}_\nu= O(t^{c(\jj)})$
maps bijectively onto $X^{[\jj]}$. 
\end{proposition}
\begin{proof}
  According to Lemma \ref{lem:cutbysyz}, we must show that
  if any matrix $\sS \in \Hom_{\oO_{X}}(\oO_{{X}}^{\oplus m},\oO_{{X}}^{\oplus n})$ 
  satisfies
  $\gr \sS = S$ and $\gG_\lambda \circ \sS = 0$, then 
  there is a unique matrix of the above form 
  such that $\gG_\lambda \circ \sS_\nu = O(t^{c(\jj)})$. 

  First we check uniqueness.  Suppose given $\lambda, \nu, \nu'$ 
  such that 
  $\mathcal{G}_\lambda\circ \mathcal{S}_{\nu}=
  \mathcal{G}_\lambda\circ\mathcal{S}_{\nu'}=0 \mod t^{c(\jj)}$.  
  On the one hand, the columns of 
  $gr(\mathcal{S}_{\nu'}-\mathcal{S}_{\nu})$ are in 
  the subspace 
  $W'_\jj \subset W_{\jj} = \oplus_j
  \CC[\Gamma]_{< c(\jj) - \gamma_j}$ spanned by elements
  of the form $t^{\rho(s)} e_{g(s)}$, where $e_{g(s)}$ is the unit
  vector in the $g(s)$-th summand.  By inspection
  $\ker G \cap W'_\jj = 0$. On the other hand the
  columns of $gr(\mathcal{S}_{\nu'}-\mathcal{S}_{\nu})$ 
  are necessarily in the kernel of $G$. 

  For existence, it suffices to observe 
  that $G(W'_\jj)=G(W_\jj)=\jj_{<c(\jj)}$. Indeed, we can compute $S_{\nu}$ by induction:
  let $\{ j_1<j_2<\dots<j_N\}=\jj\cap [0,c(\jj)]$ and suppose we found $\nu_k$
  such that $$ v_k:=\mathcal{G}_\lambda\circ\mathcal{S}_{\nu_k}=0 \mod t^{j_k}.$$
  By assumption $\lambda\in X^{[\jj]}$ so  the entries of the vector $v_k$ have orders which are elements of \(\jj\).
   Hence by the construction 
  of $W'_\jj$ there is $\nu'\in Syz$ such that ${g}r (\mathcal{G}_\lambda\circ \mathcal{S}_{\nu'})=
  {gr}(v_k).$ Thus for $\nu_{k+1}:=\nu_k-\nu'$ we have
  $$ \mathcal{G}_\lambda\circ\mathcal{S}_{\nu_{k+1}}=0 \mod t^{j_{k+1}}.$$
\end{proof}

For future use we write the matrix entries 
$(\mathcal{G}_\lambda\circ\mathcal{S}_\nu)_i  = 
\sum_j (\mathcal{G}_\lambda)_j (\sS_\nu)^j_i $ explicitly:

\[
\sum_j
\left(u_i^j\phi_{\sigma_i-\gamma_j} \phi_{\gamma_j}  + \!\!\!\!\!\!\!\!
\sum_{\substack{c(\jj)-\gamma_j>s>\sigma_i \\ g(s) = j}}   \!\!\!\!\!\!
\nu_{is}^{s-\sigma_i}\phi_{s - \gamma_j} \phi_{\gamma_j}
+
\!\!\!\!\sum_{k\in\Gamma_{>\gamma_j}\setminus\jj}\!\!\!
u_i^j \lambda^{k-\gamma_j}_j \phi_{\sigma_i-\gamma_j} \phi_k \, +
\!\!\!\!\!\!\!\sum_{\substack{c(\jj)-\gamma_j>s>\sigma_i \\ g(s) = j \\ k\in\Gamma_{>\gamma_j}\setminus\jj}}\!\!\!\!\!\!
\nu_{is}^{s-\sigma_i} 
\lambda^{k-\gamma_j}_j \phi_{s - \gamma_j} \phi_k \right)
\]

Let $\mathcal{I} \subset \CC[\lambda_\bullet^\bullet,\nu_{\bullet,\bullet}^\bullet]$ 
be the ideal defining the Hilbert scheme $X^{\jj}$, i.e., the ideal
generated by the entries of $\gG_\lambda \circ \sS_\nu$.
Then we write $X^{[\jj]}_{<r} \subset 
\mathrm{Spec}\, \CC[\lambda^{<r}_\bullet,\nu^{<r}_{\bullet,\bullet}]$ for the subscheme cut out by the ideal 
$\mathcal{I}^{<r} := \mathcal{I} \cap \CC[\lambda^{<r}_{\bullet},\nu^{<r}_{\bullet,\bullet}]$. 

Expand $(\mathcal{G}_\lambda\circ\mathcal{S}_\nu)_i$ in the basis
$\phi_k$, and denote by $Eq^r_i$ the coefficient of 
$\phi_{r+\sigma_i}$.
Implicitly $Eq^r_i$ does not occur
if $r+\sigma_i\notin \Gamma$ or $r+\sigma_i\ge c(\jj)$.
Accounting for the upper triangularity of the basis $\phi_k$, 
we see the nontrivial equations $Eq_{i}^r$  are of the following form:

\begin{gather*}
L^r_i+\mbox{  terms in } 
\lambda^{<r},\nu^{<r},\\ 
L^r_i:=\delta_{\jj\cap (\sigma_i,c(\jj))}(r+\sigma_i) \nu_{i,r+\sigma_i}^{r}+\sum_{j=1}^n
 \delta_{\Gamma\setminus\jj}(r+\gamma_j) u_i^j\lambda_j^{r}.
\end{gather*}

As
$\mathcal{I}^{<r+1}=(\mathcal{I}^{<r}, Eq_1^r,\dots,Eq_m^r)$, the space
$X^{[\jj]}_{<r+1}$ is cut out of $X^{[\jj]}_{<r} \times \mathrm{Spec}\, \CC[\lambda^{r},
\nu^{r}]$ by the ideal $(Eq_1^r,\dots,Eq_m^r)$.
We write $\pi_r:X^{[\jj]}_{<r+1} \to X^{[\jj]}_{<r}$ for the projection.
Once $r$ is greater than the 
conductor $c(\jj)$, we have
$X^{[\jj]}_{<r} = X^{[\jj]}$ by the equivalence of (1) and (2) of Lemma
\ref{lem:cutbysyz}.  Thus it remains to study the fibers in the sequence
\[ X^{[\jj]} \xrightarrow{\pi_{c(\jj)}}  X^{[\jj]}_{<c(\jj)} \xrightarrow{\pi_{c(\jj)-1}} 
\cdots \xrightarrow{\pi_0} X^{[\jj]}_{<0} = \mathrm{pt} \]

For general singularities, the projections are  not surjective and the fibers are hard to control. 
However, when the linear forms $L_{i}^r$ are independent, the fibers are affine 
spaces of constant dimension and the projections are surjective. 
We show next that this independence holds when the singularity is 
unibranch with a single Puiseux pair, and 
compute the dimensions of fibers.

\subsection{$\Gamma=\langle n, k\rangle$}

We now restrict ourselves to singularities $X$ with semigroup 
$\Gamma$ generated by two relatively
prime integers, $n$ and $k$.  The prototypical example is $x^n = y^k$,
but in fact such singularities vary with moduli \cite{ZT}. 
We will show in these cases that the Hilbert
schemes are stratified by affine cells. 

We often picture $\Gamma$ in terms of the coordinate plane, in which
we write $ni+kj$ in the unit square with bottom-left coordinate $(i,j)$. 
As every element $m \in \Gamma$ admits a unique presentation 
in the form $m = ak + bn$ where $0 \le a < n$, each occurs exactly once
in the semi-infinite strip of height $n$ in the first quadrant.  Ideals 
correspond to Young diagrams (or staircases) contained in the strip containing at most
$k$ columns of height strictly between $0$ and $n$.  For example,  
we assign the following staircase to $(t^{21}, t^{23}, t^{24}) \subset
\CC[[t^4, t^5]]$.  Bold numbers correspond to monomials in the ideal.

\begin{center}
\begin{tabular}{|c|c|c|c|c|c|c|c}
  \hline
  {\tiny 15} & {\tiny 19} & {\bf 23} & {\bf  27}  & {\bf 31} & {\bf 35} 
  & {\bf 39} & {\bf 43} \\
  \hline
  {\tiny 10} & {\tiny 14} & {\tiny 18} & {\tiny 22} & 
  {\bf 26} & {\bf 30} & {\bf 34} & {\bf 38} \\
  \hline
  {\tiny 5}  & {\tiny 9}  & {\tiny 13} & {\tiny 17} & 
  {\bf 21} & {\bf 25} & {\bf 29} & {\bf 33} \\
  \hline
  {\tiny 0}  & {\tiny 4}  & {\tiny 8}  & {\tiny 12} & 
  {\tiny 16} & {\tiny 20} & {\bf 24} & {\bf 28} \\
  \hline 
\end{tabular}
\end{center}

To see the generators and syzygies it is better to draw the infinite staircase
of which the elements in the ideal are above and those not in the ideal are below.
For example in the above case we get:

\begin{center}
\begin{tabular}{|c|c|c|c|c|c|c|c|c|c}
  \hline
  {\tiny 20}&{\bf 24}&{\bf 28}&{\bf 32}&{\bf 36}&{\bf 40}&{\bf 44}&{\bf 48}&{\bf 52}&{\bf 56}\\
  \hline
  {\tiny 15} & {\tiny 19} & {\bf 23} & {\bf  27}  & {\bf 31} & {\bf 35} 
  & {\bf 39} & {\bf 43} &{\bf 47}&{\bf 51} \\
  \hline
  {\tiny 10} & {\tiny 14} & {\tiny 18} & {\tiny 22} & 
  {\bf 26} & {\bf 30} & {\bf 34} & {\bf 38}&{\bf 42}&{\bf 46} \\
  \hline
  {\tiny 5}  & {\tiny 9}  & {\tiny 13} & {\tiny 17} & 
  {\bf 21} & {\bf 25} & {\bf 29} & {\bf 33}&{\bf 37}&{\bf 41} \\
  \hline
  {\tiny 0}  & {\tiny 4}  & {\tiny 8}  & {\tiny 12} & 
  {\tiny 16} & {\tiny 20} & {\bf 24} & {\bf 28}&{\bf 32}&{\bf 36} \\
  \hline 
  {\tiny -5}&{\tiny -1}&{\tiny 3}&{\tiny 7}&{\tiny 11}&{\tiny 15}&{\tiny 19}&{\bf 23}&{\bf 27}&{\bf 31}\\
  \hline
  {\tiny -10}&{\tiny -6}&{\tiny -2}&{\tiny 2}&{\tiny 6}&{\tiny 10}&{\tiny 14}&{\tiny 18}&{\tiny 22}&{\bf 26}\\
  \hline
  {\tiny -15}&{\tiny -11}&{\tiny -7}&{\tiny -3}&{\tiny 1}&{\tiny 5}&{\tiny 9}&{\tiny 13}&{\tiny 17}&{\bf 21}\\
  \hline
\end{tabular}
\end{center}

The generators of the ideal are the concave corners
of the staircase.  We order them $(\gamma_1,\dots,\gamma_m)$ 
in such a way that if one began ascending the staircase
at $\gamma_1$ then one would encounter them in order. 
For example in the above we may take $\gamma_1=24,\gamma_2=21,\gamma_3=23$.  
 Moreover we take the labels
modulo $m$, i.e., $\gamma_k = \gamma_{m+k}$.  
We write $\sigma_i$ for the outside the convex corner encountered between 
$\gamma_i$ and $\gamma_{i+1}$.  That is, writing $\gamma_i = a_i k + b_i n$ 
with $0 \le a_i < p$, we define $\sigma_i=a_{i+1} q+b_{i}p.$
In the example above we have $\sigma_1=29,\sigma_2=31,\sigma_3=28$. 
It is obvious from the pictorial description that 
the module of syzygies of $(t^{\gamma_1},\dots,t^{\gamma_m})$ is generated 
over $\CC[\Gamma]$ by the elements 
$\omega_i=e_i t^{\sigma_i-\gamma_i}-e_{i+1} t^{\sigma_i-\gamma_{i+1}}$ for $i=1,\dots,m.$

\begin{theorem} \label{thm:unnesteddims}
For $\jj = (\gamma_1,\ldots,\gamma_m) \subset \Gamma = \langle n, k \rangle$,
$X^{[\jj]}$ admits a bijective morphism from $\AA^{N(\jj)}$ where 
$$N(\jj)=\sum_i |\Gamma_{>\gamma_i}\setminus\jj|-\sum_i |\Gamma_{>\sigma_i}\setminus\jj|.$$
\end{theorem}
\begin{proof}
  We will study the maps 
  $X^{[\jj]}_{<r+1} \to X^{[\jj]}_{<r}$. 
  We have seen that $X^{[\jj]}_{<r+1}$ is cut out of 
  $\Spec\,\CC[\lambda^r_{\bullet}, \nu^r_{\bullet,\bullet}]$
  by some equations $Eq_i^r$, where by our description of the syzygies 
  $i = 1,\ldots,r$.  We have written $L_i^r$ for the linear term of $Eq_i^r$.  
  Once we show that $L_i^r$ are linearly independent, it will follow that the 
  zero locus of the $Eq_i^r$ is isomorphic to the zero locus of the $L_i^r$.
  Thus $X^{[\jj]}_{<r+1}$ is a vector bundle over $X^{[\jj]}_{<r}$ of fiber dimension
  \begin{align*}
   \lambda^m_{\bullet} + \# \nu^r_{\bullet,\bullet} - \# Eq_i^r & = 
    \sum_{i=1}^m \delta_{\Gamma\setminus\jj}(r+\gamma_i) +
  \sum_{i=1}^m\delta_{\jj\cap(0,c(\jj))}(r+\sigma_i) - 
  \sum_{i=1}^m\delta_{\Gamma \cap(0,c(\jj))}(r+\sigma_i) \\
   & = 
  \sum_{i=1}^m \delta_{\Gamma\setminus\jj}(r+\gamma_i)-
  \sum_{i=1}^m\delta_{\Gamma\setminus\jj}(r+\sigma_i).
  \end{align*}
  Summing over $r$ gives the claimed value of $N(\jj)$. 
  
  It remains to prove the linear independence of the forms $L_i^r$. 
  If $r+\sigma_i\notin \jj\cap(\sigma_i,c)$ then $L_i^r$ does not 
  depend on variables $\nu_{\bullet,\bullet}^r$.  On the 
  other hand if $r+\sigma_i\in \jj\cap(\sigma_i,c)$ then $L_i^r$ does depend
  on $\nu_{\bullet,\bullet}^r$. Moreover, the linear forms $L_i^r$, 
  $r+\sigma_i\in\jj\cap(\sigma_i,c)$ are mutually linearly independent
  and linear span of these forms does not contain a nontrivial linear form independent
  of $\nu_{\bullet,\bullet}^m$.
  Thus to finish the proof we need to show that the linear forms:
  $$ L_i^r,\quad r+\sigma_i\in\Gamma\setminus\jj$$
  are linearly independent.
  The easiest proof of this statement uses a pictorial interpretation of the syzygies and the linear forms 
  $L_i^r$ which we explain below.
 
Suppose we are given $i$ and $r>0$ such that $r+\sigma_i\in \Gamma\setminus\jj$. 
From our description of the syzygies
we see that $L^r_i$ is of the form:
$$ \lambda^r_{i+1}\delta_{\Gamma\setminus\jj}(\gamma_{i+1}+r)-
\lambda^r_i\delta_{\Gamma\setminus\jj}(\gamma_i+r).$$
We can visualize the nontrivial coefficients of the \(L^r_i\) in the following way. 
Begin with an \(m\times m\) matrix \(A\) with \(A_{ii}=-1\) and \(A_{i+1,i}=1\) (where the latter equation is to be interpreted \(\mod m\)). If \(r + \sigma_i \not \in \Gamma \setminus \jj\), we delete the \(i\)th column of \(A\). Finally, we replace some of the nonzero coefficients with \(0\), depending on the value of \( \delta_{\Gamma\setminus\jj}(\gamma_i+r)\). 

To interpret the condition \(r + \sigma_i  \in \Gamma \setminus \jj\), write
$r=\alpha k+\beta n$ and let \(v\) be the vector \((\alpha,\beta)\). Consider the new infinite staircase obtained by translating the original staircase by \(v\). 
 Then  \(r + \sigma_i  \in \Gamma \setminus \jj\) if and only if the translate of the convex corner corresponding 
to $\sigma_i$ shifts to a square below the original staircase  and above the (rather simpler) staircase defined by the elements of \(\Gamma\). 

Now if \(r + \sigma_i  \in \Gamma \setminus \jj\) for all \(i\),  the entire translate of the infinite staircase by \(v\) would lie below the infinite staircase. This would imply \(\jj \subset \jj+r\), which is impossible, since \(r>0\). Thus there is at least one deleted column in the matrix \(A\). Without loss  of generality, we may renumber the \(\gamma_i\) so that the final column is deleted. The remaining matrix is supported on the diagonal and on the off-diagonal just below it. After deleting additional columns, the remaining matrix will be block diagonal. The number of blocks will be the number of ``runs" of consecutive undeleted columns, and the size of each block will be   \((i+1)\times i\), where \(i\) is the length of the run. 

More geometrically, each block corresponds to a  connected component of 
the intersection of the shifted infinite staircase with the region below the original staircase and above the \(\Gamma\)-staircase.  For example, in the case of the picture below, the matrix 
consists of three blocks of sizes: $i=1,2,1$.

\begin{center}
 \includegraphics[scale=0.3]{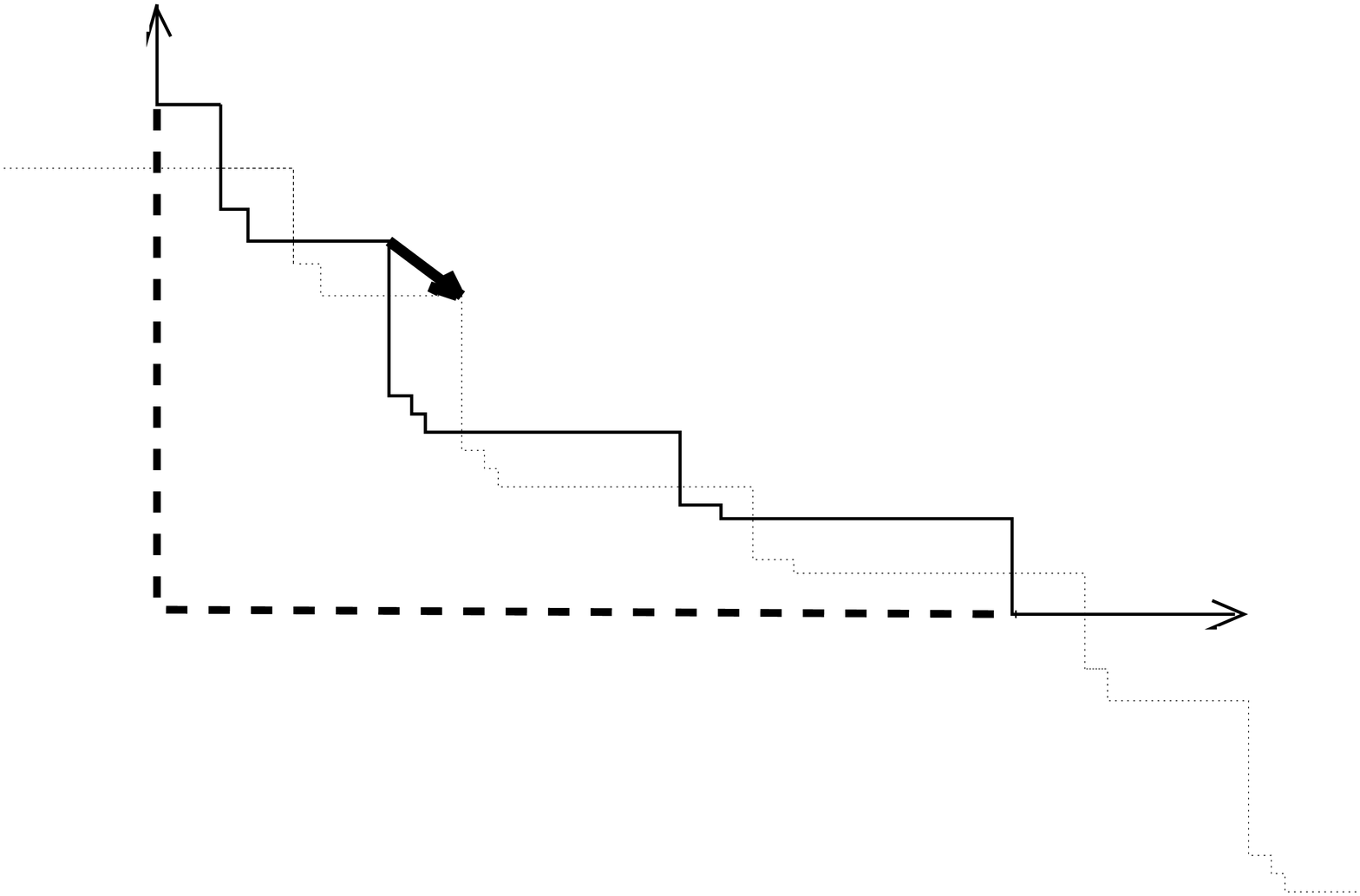}
\end{center}
Now we need to show that the diagonal blocks are of maximal rank. There are four types of diagonal block, depending on how the path exits the finite stair case. These four types are shown in the figure below. As in the previous figure, the fine dotted line marks the translated copy of the infinite staircase, the heavy solid line is the finite staircase, and the heavy dashed lines mark the boundary of the first quadrant. 
\begin{center}
 \includegraphics[scale=0.7]{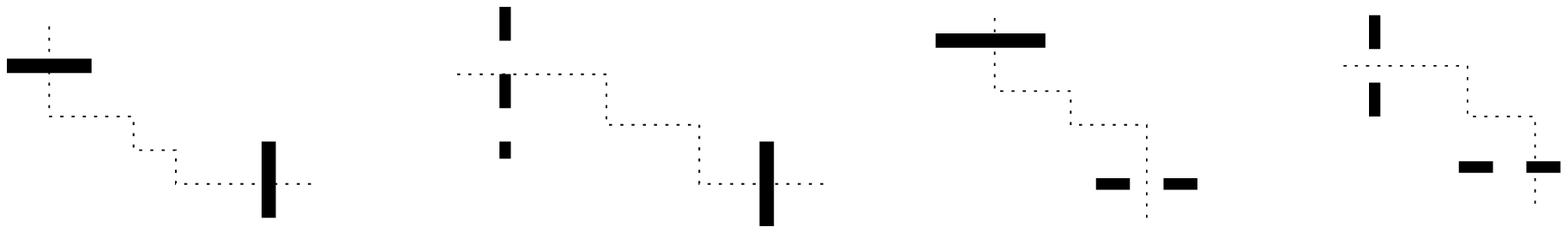}
\end{center}
In the first three cases the corresponding diagonal 
blocks are of maximal rank but in the last case the matrix
is degenerate. To finish the proof, we must check that the last case does not occur.
Indeed, if it did, we see from the picture  that  the fine dotted curve would have a unique connected component  between
the heavy dotted curve and the bold curve. Thus the fine dotted curve would lie under the bold curve, but we showed above 
that this is impossible. 
\end{proof}

\begin{remark}
Above argument actually shows that $X^{[\jj]}\simeq\mathbb{A}^{N(\jj)}$. Indeed, in the proof we show that 
$\pi$: $X^{[\jj]}_{<r+1}\to X^{[\jj]}_{<r}$ is a smooth affine fibration. But by construction the fibration
is a subfibration of the trivial fibration and the linear part of the equations defining the subfibration
does not depend on the point on $X^{[\jj]}_{<r}$. Thus $\pi$ is a trivial fibration.
\end{remark}

\subsection{Nested Hilbert schemes}

We turn to the study of the nested Hilbert schemes. 
Let $M$ be the maximal ideal of $\oO_X$. 
Recall that $X^{[a \le b]} \subset X^{[a]} \times X^{[b]}$ was defined to
be the locus 
$\{(J,I) \in X^{[a]} \times X^{[b]} \, | \, 
J \supset I \supset MJ \}$.  This space admits a stratification
by type of semigroup-module.  Indeed let 
$X^{[\jj \supset \ii]} \subset X^{[\jj]} \times X^{[\ii]}$
be the locus of $\{(J,I) \in X^{[\jj]} \times X^{[\ii]}\,|\,
J \supset I \supset M J \}$.  Evidently 
\[X^{[a \le b]} = \coprod_{\substack{\#\Gamma \setminus \jj = a \\
\#\Gamma \setminus \ii = b}} X^{[\jj \supset \ii]}\] 

For the source of the nested version of the ``exponential'' map, we take an
affine space $Gen$ with coordinates $\lambda_i^{k-\gamma_i}$
where $k \in \Gamma_{> \gamma_i} \setminus \jj$ if $\gamma_i \notin \ii$,
and $k \in \Gamma_{> \gamma_i} \setminus \ii$ if $\gamma_i \in \ii$.  
We define as before
\[\tau_{\gamma_j}(\lambda^\bullet_\bullet) := \phi_{\gamma_j} + \sum_{s} \lambda_j^{s-\gamma_j} 
  \phi_{s-\gamma_j} \]
and set 
\begin{eqnarray*}
  J(\lambda^\bullet_\bullet) & = & (\tau_{\gamma_i}(\lambda^\bullet_\bullet)) \\
  I(\lambda^\bullet_\bullet) & = & (\tau_{\gamma_i}(\lambda^\bullet_\bullet)\,|\, \gamma_i \in \ii)
  + M J(\lambda^\bullet_\bullet)
\end{eqnarray*}
Finally we define 
$\Exp(\lambda^\bullet_\bullet) = (J(\lambda^\bullet_\bullet), I(\lambda^\bullet_\bullet))$. 
This constructible map induces a bijective morphism
$\Exp_\gamma^{-1}(X^{[\jj \supset \ii]}) \to X^{[\jj \supset \ii]}$. 
The locus $\Exp_\gamma^{-1}(X^{[\jj \supset \ii]})$ is characterized
by requiring $\ord(J(\lambda)) = \jj$ and $\ord(I(\lambda)) = \ii$, which is equivalent to requiring that the syzygies of $\underline{\jj}$ and 
$\underline{\ii}$  lift to $J(\lambda)$ and $I(\lambda)$ respectively.
In fact, because by construction $I(\lambda) \supset M J(\lambda)$, we see that 
\[\ord(J(\lambda)) = \ord(J(\lambda) \setminus M J(\lambda)) \cup \ord(M J(\lambda)) \subset 
\jj \cup \ord(I(\lambda)).\]
Therefore it suffices to check that $\ord(I(\lambda)) = \ii$, or in other words 
we need only concern ourselves with the syzygies of $\ii$.  

Rather than continue a general treatment, we restrict to the case
when $\Gamma = \langle n, k \rangle$ in which the syzygies are easier
to describe. 

\begin{theorem} 
\label{thm:nesteddims}
Let $\jj = (\gamma_1,\ldots,\gamma_m) \subset \Gamma = \langle n, k \rangle$, and
$\jj \supset \ii \supset \jj + \mm$.  Let $\sigma_i$ be the degrees of the syzygies
of $\jj$.  Then
$X^{[\jj\supset\ii]}$ admits a bijective morphism from $\AA^{N(\jj \supset \ii)}$ where 
$$N(\jj \supset \ii)=\sum_{\gamma_i \notin \ii}
|\Gamma_{>\gamma_i}\setminus\jj| +\sum_{\gamma_i \in \ii} |\Gamma_{>\gamma_i}\setminus\ii|-
\sum_{i=1}^m|\Gamma_{>\sigma_i}\setminus \ii|.$$
\end{theorem}
\begin{proof}
  We must determine the locus in $Gen$ in
  which the syzygies of $\ii$ lift.   Here we drop our numbering convention
  of the generators of $\jj$ and instead number them $\gamma_1,\ldots,\gamma_l,
  \gamma_{l+1},\ldots,\gamma_m$ so that $\jj \setminus \ii = \{\gamma_{1},\ldots,\gamma_l\}$. 
  Then the following is a not-necessarily-minimal set of generators for $\ii$:
  $\{\gamma_1+n, \gamma_1+k,\ldots,\gamma_l+n,\gamma_l+k\} \cup \{\gamma_{l+1},\ldots,\gamma_m\}$. 
  From this, for instance from the pictorial description, it is easy to see that
  the degrees of the minimal syzygies are among 
  $\sigma_1,\dots,\sigma_m,\gamma_1+n+k,\dots,\gamma_l+n+k$, where the $\sigma_i$ are the 
  syzygies of $\jj$.  Note that the syzygies of degree
  $\gamma_i + p + q$ between the generators of degree $\gamma_i + n$ and $\gamma_i + k$ 
  always lift to $I(\lambda)$.  Indeed, if $f_i$ is the generator of degree $\gamma_i$,
  then the generators of degrees $\gamma_i + n$ and $\gamma_i + k$ are just $\phi_n f_i$ and 
  $\phi_k f_i$,
  thus the syzygy in question is $\phi_n (\phi_k f_i) - \phi_k(\phi_n f_i) = 0$.
  
  We denote the generators of $\ii$ as follows:
  $$\epsilon_{2i-1}=\gamma_i+n,\quad\epsilon_{2i}=\gamma_i+k,\quad i=1,\dots,l,\quad \epsilon_{i+l}=
  \gamma_i \quad i=l+1,\dots,m.$$ 
  We have reduced the problem of determining when $\lambda \in X^{[\jj \supset \ii]}$ to
  determining when the syzygies of the form $(\vec{s}_i)_j=u^j_i t^{\sigma_i-\epsilon_j}$,
  $j=1,\dots,m+l$, $i=1,\dots,m$ lift.  
  These are generated by the syzygies in which only two entries $u^j_i$ are nonzero.

  As in the un-nested case, we introduce a subspace $Syz$ of all possible such
  syzygies so that $X^{[\jj \supset \ii]}$ will be cut out of $Gen \times Syz$ by explicit
  equations.  Thus fix, for each  $s \in \ii$, a splitting $s = \epsilon_{g(s)} + \rho(s)$
  with $\rho(s) \in \Gamma$.  
  Let $Syz$ be an affine space with coordinates 
  $\nu_{is}^{s-\sigma_i}$ where 
  $i=1,\dots,m$ and $c(\ii)>s>\sigma_i$. 
  To a point $\nu$ in $Syz$ we assign a $(m+l)\times m$ matrix with entries:

  $$(\mathcal{S}_\nu)_i^j=u_i^j\phi_{\sigma_i-\epsilon_j}+
  \sum_{\substack{c(\ii)>s>\sigma_i \\ g(s) = j}} 
  \nu_{is}^{s -\sigma_i}\phi_{s-\epsilon_j}.$$

  By an argument similar to that in the un-nested case, $X^{[\jj \supset \ii]}$ is cut 
  out of $Gen \times Syz$ by the equation 
  $\mathcal{G}_\lambda\circ \mathcal{S}_\nu=0 \mod t^{c(\ii)}$.  Denoting by $\iI$ 
  the ideal generated by the matrix entries of $\mathcal{G}_\lambda\circ \mathcal{S}_\nu$,
  we may again define the successive approximations 
  $X^{[\jj \supset \ii]}_{<r} \subset \CC[\lambda_{\bullet}^{<r}, \nu_{\bullet,\bullet}^{<r}]$,
  as the locus cut out by $\iI \cap \CC[\lambda_{\bullet}^{<r}, \nu_{\bullet,\bullet}^{<r}]$. 
  Then $X^{[\jj \supset \ii]}_{<r+1}$ is cut out of $
  X^{[\jj \supset \ii]}_{<r} \times \Spec\,\CC[\lambda_{\bullet}^{<r}, \nu_{\bullet,\bullet}^{<r}]$
  by the coefficients $Eq^r_i$ of $\phi_{r+\sigma_i}$ in the matrix entries of 
  $\mathcal{G}_\lambda\circ \mathcal{S}_\nu$.  The terms depending on 
  $\lambda^r_\bullet, \nu^r_{\bullet,\bullet}$ are linear, and we denote them by $L^r_i$:

\begin{equation}
Eq_i^r = L^r_i+\mbox{ terms in } \lambda^{<r},\nu^{<r},
\end{equation}
\begin{multline*}
L^r_i=\delta_{\ii\cap(\sigma_i,c(\jj))}(r+\sigma_i)\nu_{i,r+\sigma_i}^r +\sum_{s=1}^l (u_i^{2s-1}+u_i^{2s})\delta_{\Gamma\setminus\jj}(r+\gamma_s)\lambda^r_s+
\sum_{s=l+1}^m u_i^{l+s}\delta_{\Gamma\setminus \ii}(r+\gamma_s)\lambda^r_{s},
\end{multline*}
here $Eq_i^r=0$ if $r+\sigma_i\notin \Gamma$ or $r+\sigma_i\ge c(\ii)$.

  As in the proof of Theorem \ref{thm:unnesteddims}, from the shape of the forms $L^r_i$ we see
  that we only need to show that the linear forms:
  $$ L^r_i,\quad r+\sigma_i\in\Gamma\setminus\ii$$
  are linearly independent.  From the structure of the coefficients $u^i_j$ we see that 
  for $r+\sigma_i\in \Gamma\setminus\ii$ the linear form $L^r_i$ is equal to:
  $$\sum_{s=1}^l (u_i^{2s-1}+u_i^{2s})\delta_{\Gamma\setminus\jj}(r+\gamma_s)\lambda^r_s+
  \sum_{s=l+1}^m u_i^{l+s}\delta_{\Gamma\setminus \ii}(r+\gamma_s)\lambda^r_{s}.$$
  
   At this point, it is convenient to reorder the generators of \(\jj\) consecutively along the staircase, as we did in the unnested case.  Once we have done so, we see that the matrix of coefficients of the \(L_i^r\) is obtained as in the unnested case: We start with the same matrix \(A\), delete some columns (corresponding to those \(\sigma_i +r \not \in \Gamma \setminus \ii\)), and then set some coefficients to zero. The set of coefficients which is set to zero is  smaller than that in the unnested case, since the support of \(\delta_{\Gamma\setminus \ii}\) is larger than that of \(\delta_{\Gamma\setminus \ii}\). 
  
  Hence, to compute the rank of the space spanned by the linear
  forms $L^r_i$, $r+\sigma_i\in \Gamma\setminus\ii$, one has to analyze connected 
  components of the intersection of the periodic path associated to $\jj$ and shifted by a
  vector $v$ (as in Theorem \ref{thm:unnesteddims}) with the area under the path associated to $\ii$ and above the path associated to \(\Gamma\). The path $\jj+r$ cannot lie completely under the path for $\ii$: if it did, it is easy to see that we would have \(\jj+r \subset \jj\). For \(r>0\), we have already shown this is impossible. The remainder of the argument proceeds exactly as in the proof of  Theorem \ref{thm:unnesteddims}. 

  Thus we conclude that 
  $\pi_r: X^{[\jj \supset \ii]}_{<r+1}\to X^{[\jj\supset\ii]}_{<r}$ is a vector bundle
  with fibers of dimension 
  $ \# \lambda^r_{\bullet} + \# \nu^r_{\bullet,\bullet} - \# Eq^r_i$.  
  By construction
  \begin{eqnarray*}
  \# \lambda^r_{\bullet} & = & \sum_{i=1}^l\delta_{\Gamma\setminus\jj}(r+\gamma_i)+
    \sum_{i=l+1}^m\delta_{\Gamma\setminus\ii}(r+\gamma_i) \\
  \# \nu^r_{\bullet,\bullet} & = &  \sum_{i=1}^m \delta_{\ii\cap (0,c(\ii))}(r+\sigma_i)
  \end{eqnarray*}
  while on the other hand the nontrivial $Eq^r_i$ impose  
  $$   \# Eq^r_i  = \sum_{i=1}^m \delta_{\Gamma \cap (0,c(\ii))}(r+\sigma_i) $$
  independent conditions.  Summing over these terms and summing over $m$ gives the 
  stated dimension. 
\end{proof}

\section{Examples}
\label{sec:examples}

In this section we verify that the predictions of Conjecture~\ref{conj:homflyhomology} agree with previously known or conjectured values of the HOMFLY homology for torus knots. 
We first consider the ``stable'' HOMFLY homology, whose Poincar{\'e} polynomial is defined by 
$$\uspp (T(n,\infty)) = \left( \frac{a}{q}\right) ^{\mu-1} \lim_{k\to\infty} \uspp(T(n,k)). $$
This limit exists by a theorem of Stosic \cite{stosic}. It was conjectured in  \cite{DGR} that 
$$\uspp (T(n,\infty))  = \prod_{i=1}^{n} \frac{1+a^2 q^{2i-2} t^{2i-1}}
{1-q^{2i}t^{2i-2}}. $$
Proposition~\ref{prop:infty} from the introduction says that the analogous statement holds for \(\usppa\). 

\begin{proof}(Of Proposition~\ref{prop:infty})
The assertion $\usppa(X_{n,\infty}) = \usppa(X_{n,k}) + O(q^{2k})$
follows immediately from the fact that all ideals parameterized by
the spaces
$X_{n,\infty}^{[\ell \le \ell +m]}$ and 
$X_{n,k}^{[\ell \le \ell + m]}$
for $\ell < k$ contain the ideal $(x,y)^k$.  But modulo this 
ideal the equations $y^n = x^k$ and $y^n = 0$ are identical, so
for $\ell < k$ we have the equality 
$X_{n,\infty}^{[\ell \le \ell +m]}= X_{n,k}^{[\ell \le \ell + m]}$.

Recall that $X_{n,k}^{[\ell \le \ell + m]}$ is a union of affine
spaces corresponding to staircases with some marked external corners 
sitting in the semi-infinite strip of height $n$ in the first quadrant.  
There is a condition on which staircases occur, but any fixed staircase
will contribute for all sufficiently large $k$.  By inspection
of the formula in Theorem \ref{thm:unnesteddims}, the contributions
converge as well.  

Let $S = (n^{s_n}, (n-1)^{s_{n-1}}, \ldots, 1^{s_1})$ 
be a staircase with $s_i$ columns of height $i$.  For any subset
$\Sigma \subset \{1,\ldots,n\}$, we form $\Sigma'$ by subtracting 
$1$ from each element of $\Sigma$.  Write $S \cup \Sigma$ for the staircase
with additional columns with heights from $\Sigma$, and similarly 
$S \cup \Sigma'$. 
Every admissible nested pair of semigroup ideals
$\jj \supset \ii$ has staircase of the form $S \cup \Sigma \supset 
S \cup \Sigma'$ 
for some $S, \Sigma$: to recover $S$ from the staircases of $\jj$ and  $\ii$, 
delete every column in which those staircases differ.

It remains to sum the contributions of the staircases.  When $k+n^2$ is greater
than the number of boxes in the staircase, one calculates from
Theorem \ref{thm:unnesteddims} that 
\[\mathrm{Cont}_S = q^{2 \sum i s_i} t^{2 \sum (i-1)s_i} 
\prod_{i=1}^n (1+a^2 q^{2i-2} t^{2i-1})\]
Summing over staircases gives the stated formula. 
\end{proof}

 When combined
with the symmetry of \(\usppa\), Proposition~\ref{prop:infty} is enough to determine 
$\usppa(X_{2,n})$ and 
$\usppa(X_{3,n})$.

\begin{corollary} 
 \begin{align*}
\frac{1 - q^2}{1+a^2 t}
\left( \frac{q}{a}
\right)^{2k - 2} \usppa(X_{2,2k+1}) & =  
 \frac{1-t^{2k+1}q^{4k+2}}{1-t^2q^4} + a^2 q^2 t^3 
\frac{1-t^{2k-1}q^{4k-2}}{1-t^2 q^4}   \\
  \frac{1 - q^2}{1+a^2 t}
\left( \frac{q}{a}
\right)^{2k - 3}
 \uspp_{\mathrm{alg}}(X_{2,2k})\phantom{_{+1}} & =  
\frac{1+a^2 q^2 t^3}{1-q^4 t^2} - q^{2+4k} t^{2+2k} \frac{q^2+a^2 t}{1-q^4 t^2}
 \end{align*}
 
 \end{corollary}
 
 \begin{corollary}

\begin{eqnarray*}
\frac{1 - q^2}{1+a^2 t}
\left( \frac{q}{a}
\right)^{6k - 1}
\uspp_{\mathrm{alg}}(X_{3,3k+1}) & = &
\frac{(1+a^2 q^2 t^3)(1+a^2 q^4 t^5)}{(1-q^4 t^2)(1-q^6 t^4)} \\
& - & q^{2+6k} t^{2+4k} \frac{(q^2 + a^2 t)(1+a^2 q^2 t^3)(1+q^2 t^2 + q^4 t^2)}
{(1-q^6 t^2)(1-q^6 t^4)} \\
& + & q^{4+12k}t^{4+6k} \frac{(q^2+a^2t)(q^4+a^2 t)}{(1-q^6 t^2)(1-q^4 t^2)}
\end{eqnarray*}
\begin{eqnarray*}
\frac{1 - q^2}{1+a^2 t}
\left( \frac{q}{a}
\right)^{6k + 1}
\uspp_{\mathrm{alg}}(X_{3,3k+2}) & = &
\frac{(1+a^2 q^2 t^3)(1+a^2 q^4 t^5)}{(1-q^4 t^2)(1-q^6 t^4)} \\
& - & q^{4+6k} t^{4+4k} \frac{(q^2 + a^2 t)(1+a^2 q^2 t^3)(1+q^2 + q^4 t^2)}
{(1-q^6 t^2)(1-q^6 t^4)} \\
& + & q^{8+12k}t^{6+6k} \frac{(q^2+a^2t)(q^4+a^2 t)}{(1-q^6 t^2)(1-q^4 t^2)}
\end{eqnarray*}
\end{corollary}
\begin{proof}
  We recall that $\delta(X_{2,n}) = (n-1)/2$ and
  $\delta(X_{3,n}) = n-1$. Thus by Proposition \ref{prop:infty}, 
  the equality in each case holds modulo $q^{2\delta+2}$.  By inspection,
  the RHS is a polynomial in $q$ of degree $4 \delta$; Proposition 
  \ref{prop:rands} implies the same for the LHS.  
  It remains only to observe that the symmetry
  imposed by 
  Proposition \ref{prop:rands} on the LHS also holds on the right.  
\end{proof}

 In the 
$(2,n)$ case, these formulas match the known value of
$\uspp$ \cite{KhSoergl}; in the $(3,n)$ case, the resulting formula matches
a conjectural formula for $\uspp$ advanced in \cite{DGR}.

\section{Torus knots and Cherednik algebras} 
\label{sec:daha}

In this section, we make some conjectures about the structure of \(\usppa(X_{n,k})\). The main thrust of these conjectures is that the HOMFLY homology of torus knots should be related to the representation theory of rational Cherednik algebras. We will explore this idea more fully in a subsequent paper \cite{GORS}. Here, we  focus on the problem of writing explicit formulas for \(\usppa(X_{n,k})\).
 
For the moment,  we restrict attention to the non-nested Hilbert scheme series \(\usppa^{\rm min}\).
As we explained in section~\ref{sec:BPS},   \(\usppa^{\rm min}(X)\) admits an 
alternate description in terms of the
perverse filtration on the 
cohomology of the compactified Jacobian of a complete rational curve
with a unique singularity of type \(X\).
When \(X\) is unibranch of the form $y^n = x^k$, 
it can be shown using the techniques of Z. Yun's thesis 
that the cohomology of the compactified Jacobian 
furnishes the unique simple representation of  $\hH^{sph}_n(k/n)$, the 
spherical 
rational Cherednik algebra of type $A_{n-1}$ and central character 
$k/n$ \cite{Y, OY}\footnote{
Varagnolo and Vasserot have shown that the equivariant K-theory of this space
admits the action of the double affine Hecke algebra, and presumably
an application of their methods in Borel-Moore homology would yield 
a construction of its trigonometric degeneration.  We however have been unable
to use their approach to construct the rational 
Cherednik algebra representation geometrically.
}.  
It is moreover the case that the perverse
filtration on $\mathrm{H}^*(\overline{J})$ is compatible 
with the natural bigrading on $\hH^{sph}_n(k/n)$; i.e., it is a good
filtration in the sense of Gordon and Stafford \cite{GS, GS2}.  
An a-priori different good filtration on $\hH^{sph}_n(k/n)$ can be constructed 
by means of results  Calaque, Enriquez, and Etingof \cite{CEE} (see details below).  
We conjecture 
these filtrations have at least the same associated graded dimensions. 
When $n = mk+1$, a formula for the character of the filtration of \cite{CEE} 
is established by the work of Gordon and Stafford \cite{GS, GS2}, which 
gives a prediction for the lowest order coefficient of $a$ in the superpolynomial
of the $(n, mn+1)$ torus knot. 

\subsection{Filtrations on rational Cherednik algebras}

We recall some relevant definitions. The symmetric group $S_n$ 
acts on the free algebra 
$\CC\langle x_1,\ldots,x_n, y_1,\ldots,y_n\rangle$ by permuting
the variables. 

\begin{definition}
For any complex number $\xi$, the 
rational Cherednik algebra $\hH_n(\xi)$ is the quotient 
of the algebra $\CC[S_n]\ltimes \CC\langle x_1,\dots,x_n,y_1,\dots,y_n\rangle $ by the 
following relations:
\begin{gather*}
\sum_{i=1}^n x_i=0=\sum_{i=1}^n y_i, \quad [x_i,x_j]=0=[y_i,y_j],\\
[x_i,y_j]=\frac1n- \xi s_{ij},\quad i\ne j.
\end{gather*}
Using the symmetrizer \[e = \frac{1}{n!}\sum_{\sigma\in S_n} \sigma \]
we form the spherical rational Cherednik algebra $\hH^{sph}_n(\xi):=e \hH_n(\xi)e$.
\end{definition}

We fix the following grading on $\hH_n(\xi)$ and  $\hH^{sph}_n(\xi)$:
\[\deg(x_i)=1,\quad \deg(y_i)=-1,\quad \deg(\sigma)=0,\quad \sigma\in S_n.\]
The grading is equivalently given by the eigenvalue of the operator
$[h,\cdot]$, where $h=\sum_i x_iy_i$.

In studying the representation theory of $\hH_n(\xi)$,
it is natural to restrict to the category 
$\mathcal{O}$ of $\hH_n(\xi)$-modules which are locally nilpotent 
under the action of operators $y_i$ and decompose 
into a direct sum of finite dimensional 
generalized eigenspaces of $h$.  We recall from \cite{BEG1,BEG2,BEG3,GGOR} some
basic facts about this abelian category. 
The simple objects are enumerated by Young diagrams of size $n$.  Denoting by $L_\xi(\mu)$ the
representation corresponding to the Young diagram $\mu$, the subspace annihilated
by the $y_i$ furnishes the irreducible representation of $S_n$ corresponding to $\mu$. 
For \(\xi>0\),  $L_\xi(\mu)$ is finite dimensional if and only if $\mu$ corresponds to 
the trivial representation and 
$\xi=k/n$ with $k\in\ZZ$ and $(k,n)=1$.  We denote this module by $L_\xi$.
One may analogously define the category of $\oO^{sph}$ of $\hH_n^{sph}(\xi)$-modules; 
the map $M\mapsto eM$ is an equivalence between $\oO$ and $\oO^{sph}$ 
for all $\xi$, except rational $\xi$ from $(-1,0)$ with denominator smaller
or equal $n$ \cite{BE}. 

The element $h$ gives a $\ZZ$-grading to elements $M$ of $\oO$ or $\oO^{sph}$: 
$$M_i=\{m\in M| h(m)=i m\}.$$
\begin{definition} \cite{GS, GS2} Let $M$ be a module in 
$\oO$ (resp. $\oO^{sph}$). A filtration 
\[
 \Lambda:\quad \Lambda^{i,j}M\subset \Lambda^{i+1,j} M,\quad \Lambda^{i,j}M\subset \Lambda^{i,j+1}M
\]
is good if 
\begin{gather*}
  z\cdot M_{ij}\subset M_{i+k,j+l}, \quad \mbox{ for any } z \mbox{ with degree } k \mbox{ in the } x\mbox{'s and degree } l \mbox { in the } y \mbox{'s}, \\
  M_k=\oplus_{i-j=k} M_k \cap\Lambda^{i,j}M,\mbox{ and } M_k\cap \Lambda^{i,j}M=0 \mbox{ if } i-j<k,
\end{gather*}
and the corresponding associated graded module is a finitely generated 
$\mathbb{Z}\times\mathbb{Z}$-graded module over 
$\CC[S_n] \ltimes \CC[x_1,\ldots,x_n,y_1,\ldots,y_n]$ 
(resp. $\CC[x_1,\ldots,x_n,y_1,\ldots,y_n]^{{S}_n}$).
\end{definition} 

The procedure of \cite{OY} realizing $L_{k/n}$ as the cohomology of the
compactified Jacobian of a curve with singularity $X^n = Y^k$ matches
the {\em grading} on $L_{k/n}$ to the associated graded structure
of the perverse filtration on the Jacobian; the process of taking
associated graded destroys the homological grading on the Jacobian,
but leaves a (good) filtration on $L_{k/n}$.  Another,  {\it a-priori} different, good
filtration is given by the following construction of \cite{CEE}:

\begin{proposition} \label{prop:CEE}
The representation $eL_{k/n}$ of $\hH^{sph}_n(k/n)$ has
a good filtration.
\end{proposition}
\begin{proof}
First, let us notice that $e\hH_n(k/n)e-\hH_n(k/n)$ bimodule $e\hH_n(k/n)$ has a good filtration.
Since this bimodule provides a Morita equalence between $e\hH_n(k/n)e$ modules and $\hH_n(k/n)$-modules,
it is enough to construct a good filtration on the spherical modules.

According to \cite[Thm. 9.8]{CEE}, there is an isomorphism of vector spaces 
$eL_{k/n} \cong eL_{n/k}$. 
Moreover, this isomorphism identifies 
$\sum_{i=1}^k x_i^a y_i^b \in \hH^{sph}_k(n/k)$ and 
$\sum_{i=1}^n x_i^a y_i^b \in \hH^{sph}_n(k/n)$ for all positive $a,b$ as 
endomorphisms of this vector space. 
The spherical subalgebra is spanned by these elements hence 
this isomorphism carries good filtrations to good filtrations.

As explained in \cite{BEG2}:
$$ eL_{\frac{k}{n}+1}=eL_{k/n}\otimes_{e\hH_n(k/n)e} P_{k/n},\quad P_{k/n}:=e 
\hH_n(\mbox{$\frac{k}{n}$}+1)\delta e,$$
where $\delta:=\prod_{i<j} (x_i-x_j)$. In particular, it is shown in \cite{BEG2} that 
$P_{k/n}$ has the structure of an  $e\hH_n(k/n)e-e\hH_n(\frac{k}{n}+1)e$ bimodule. Thus if $eL_{k/n}$ has a good
filtration, the above construction induces a good filtration on $eL_{\frac{k}{n}+1}$.

Thus by the Euclidean algorithm, we may construct a good filtration on any $L_{k/n}$ 
starting from a good filtration on $L_1$.  This latter module is one dimensional,
so we may give it the trivial filtration: $\Lambda^{0,0}L_1=0,\Lambda^{>0,>0}=L_1$.  
\end{proof}

\begin{conjecture} \label{conj:filtrations}
 The perverse filtration and the filtration of Proposition \ref{prop:CEE} agree.
\end{conjecture}

\begin{remark}
  Let $\overline{\hH}_n(k/n)$ be the quotient of 
  $\CC[S_n]\ltimes \CC[\overline{x}_1,\dots,\overline{x}_n,\overline{y}_1,\dots,\overline{y}_n]$
  by the relations
  \begin{gather*}
  [\overline{x}_i,\overline{x}_j]=[\overline{y}_i,\overline{y}_j]=0 \\
  [\overline{y}_i,\overline{x}_j]=k s_{ij},\quad i\ne j,\\
  [\overline{y}_i,\overline{x}_i]=1/(n-1)-k\sum_{j\ne i} s_{ij}.
  \end{gather*}
  Let $\overline{X}:=\sum_i \overline{x}_i$ and $\overline{Y}:=\sum_i \overline{y}_i$.  Then
  $\CC[\overline{Y},\overline{X}]$ is isomorphic to the algebra of 
  differential operators on $\CC$, 
  and $\hH_n(k/n)$ is embedded via:
  $$ x_i\mapsto \overline{x}_i-\overline{X}/n,\quad y_i\mapsto \overline{y}_i-\overline{Y}/n.$$
  Moreover, $\hH_n(k/n)$ commutes with $\overline{X}$ and $ \overline{Y}$, so 
  $\overline{\hH}_n(k/n)=\hH_n(k/n)
  \times \CC[\overline{Y},\overline{X}]$.
  Given an $\hH_n(k/n)$ module $M$ we can produce an $\overline{\hH}_n(k/n)$ module just by
  tensoring over $\CC$ with $\CC[\overline{X}]$.   It is natural
  to expect that an action of $\overline{\hH}_n(k/n)$ may be constructed on the Hilbert
  scheme homologies $\oplus_{i,n} \mathrm{H}^i(X^{[n]})$ as an algebra of correspondences,
  compatibly with the action of $\hH_n(k/n)$ on the cohomology of the compactified Jacobian.
\end{remark}

\subsection{Results of Haiman \cite{haiman2} and Gordon and Stafford \cite{GS,GS2}}

Let
$\mathrm{Hilb}(n)$ be the locus in the Hilbert scheme of $n$ points
in $\CC^2$ with center of mass at the origin.  Note that
$\CC^2 \times \mathrm{Hilb}(n)$ is the usual Hilbert scheme of $n$ points
in $\CC^2$.  There is a functor:

$$ \Phi: e\hH_n (\xi) e\mbox{-module } M\mbox{ with a good filtration } 
\Lambda\to S_{\Lambda,M}\in 
\mathrm{Coh}(\mathrm{Hilb}(n)),$$
The sheaf $S_{\Lambda,M}$ may depend
on the filtration, though its support does not.  

The functors 
$\Phi$ interact well with the shift functors. Let $\tT$ be  the tautological 
rank $n$ sheaf on $\mathrm{Hilb}(n)$ whose fiber at a given point is
the structure sheaf of the corresponding subscheme. Then
\begin{equation}
\Phi (M \otimes_{e\hH_n(k/n)e} e \hH_n(\mbox{$\frac{k}{n}+1$})\delta e) = 
\Phi(M) \otimes \mathrm{det} \tT
\end{equation}
In addition, if $L_{1/n}$ is
equipped with the trivial filtration, and $\zZ \subset \mathrm{Hilb}(n)$
is the subvariety parameterizing schemes supported at the origin, then
\begin{equation}
  \Phi(L_{1/n}) = \oO_\zZ 
\end{equation}
Combining  with
vanishing results of Haiman gives:  

\begin{theorem}\cite{GS2} \label{thm:GandS} Consider the module
$eL_{m + \frac{1}{n}}$ of $e\hH_n(m+\frac{1}{n})e$ equipped with the filtration of 
Proposition \ref{prop:CEE}.  Then
\[ \Phi(eL_{m+\frac{1}{n}})= \oO_\zZ \otimes (\mathrm{det}\, \tT)^{\otimes m}  \]
Moreover, if $T_1, T_2$ are the equivariant characters of $(\CC^*)^2$,
$$\sum_{i,j} T_1^iT_2^j \dim gr_\Lambda^{i,j} eL_{m+\frac{1}{n}} =
\chi_{\CC^*\times\CC^*} (\oO_\zZ \otimes (\mathrm{det}\, \tT)^{\otimes m} ) $$
\end{theorem}

The equivariant character is computed  \cite{haiman2} by localization in 
equivariant K-theory. 
As the calculation is in any case restricted to $\zZ$ we may compute 
on the usual Hilbert scheme $(\CC^2)^{[n]}$ rather than $\Hilb(n)$.  The
fixed points are enumerated by partitions $\lambda$ of $n$. 
For a box $x$ in the diagram of a partition $\lambda$, we
use the standard notations $a(x),l(x)$ for its arm and leg, and 
$a',l'$ for the co-arm and co-leg.  We write $\lambda'$ for the dual 
partition to $\lambda$, and 
$\kappa(\lambda):= \sum_{x \in \lambda} l'(x) = \sum {\lambda'_i\choose 2}$.
 The equivariant weights are as follows:
\begin{eqnarray} \label{eq:weights}
 \tT_\lambda & = & \sum_{x \in \lambda} T_1^{l'(x)} T_2^{a'(x)} \\
 \det \, \tT_\lambda & = & T_1^{\kappa(\lambda)} T_2^{\kappa(\lambda')}  \\
 T(\CC^2)^{[n]}_\lambda & = & \prod_{x\in\lambda}
  (1-T_1^{1+l(x)}T_2^{-a(x)})(1-T_1^{-l(x)}T_2^{1+a(x)})  \\
 \oO_\zZ|_\lambda & = & \tT_\lambda (1-T_1)(1-T_2) \prod_{x\in\lambda\setminus \{(0,0)\}}
  (1-T_1^{l'(x)}T_2^{a'(x)})
\end{eqnarray}
and so in all $\chi_{\CC^*\times\CC^*} (\oO_\zZ \otimes (\mathrm{det}\, \tT)^{\otimes m} )$
is given by the following formula:
\begin{equation} \label{eq:nnmp1}
\sum_{\lambda \vdash n} T_1^{m\kappa(\lambda)} T_2^{m\kappa(\lambda')} 
\frac{(1-T_1)(1-T_2) \prod_{x\in\lambda\setminus \{(0,0)\}}
  (1-T_1^{l'(x)}T_2^{a'(x)})}{\prod_{x\in\lambda }
  (1-T_1^{1+l(x)}T_2^{-a(x)})(1-T_1^{-l(x)}T_2^{1+a(x)})} \sum_{x \in \lambda} T_1^{l'(x)} T_2^{a'(x)}
\end{equation}

\subsection{Formulas for torus knots}

The precise relation between the filtration in the construction of \cite{OY} 
and the filtration of \cite{GS,GS2} is such that the equivariant variables
$T_1, T_2$ are related to the perverse and homological $q, t$ as follows:
\begin{equation} 
T_1:=q^2,\quad T_2:=1/(q^2t^2).
\end{equation}
Combining Proposition \ref{prop:perverse} and Conjecture \ref{conj:filtrations} 
we expect: 

\begin{conjecture}
  $(at)^{-\mu} \sppa^{\mathrm{min}} (X_{k,n})  = 
   \sum_{i,j} T_1^iT_2^j \dim gr_\Lambda^{i,j} eL_{k/n}$
\end{conjecture}

From the above results of Haiman \cite{haiman2} 
and Gordon and Stafford \cite{GS,GS2}, we expect more explicitly
that $(at)^{-\mu} \sppa^{\mathrm{min}} (X_{mn+1,n})$ is given
by the formula \ref{eq:nnmp1}. 
For general $k/n$, one might hope that the sheaf $F_{k/n}$ constructed 
by Gordon and Stafford \cite{GS} is
 $\CC^* \times \CC^*$-equivariant and that moreover we again have 
$$\sum_{i,j} T_1^iT_2^j\dim gr\Lambda^{i,j}(L_{k/n})=\chi_{\CC^*\times\CC^*}(
F_{k/n})$$

This conjecture can be extended to describe the entire polynomial 
\(\sppa (X_{n,k})\).  Recall that 
$\tT$ splits as  $\tT' \bigoplus \oO_{\Hilb}$.  We write $A := a^2 t$ and
\[\Lambda^* A \tT'^*:= \sum_{i=0}^{n-1} A^i \Lambda^i \tT'^*.\]

\begin{conjecture} \label{conj:allrows}
  For any coprime $(k,n)$, there exists a $\CC^* \times \CC^*$-equivariant
  sheaf $F_{k/n}$ such that 
  $(at)^{-\mu} \sppa(X_{k,n}) = 
  \chi_{\CC^* \times \CC^*}(F_{k/n} \otimes \Lambda^* A \tT'^*)$.
\end{conjecture}
\begin{remark}
  In \cite{GORS} it will be clarified that Conjecture \ref{conj:allrows}
  amounts to an assertion
  that the coefficients of $a$ in $\mathcal{P}_{alg}(X_{n,k})$ correspond to 
  representations of hyperspherical rational Cherednik algebras.  Note the assertion
  above is consistent with $F_{k/n} \otimes \det \tT = F_{(k+n)/n}$. 
\end{remark}

The weights of $\Lambda^* A \tT'^*$ are:
\begin{gather*}
(\Lambda^* A \tT'^*)_\lambda \,\, = \!\!\!
\prod_{x\in\lambda \setminus \{(0,0)\}} \!\!\! (1+A T_1^{-l'(x)}T_2^{-a'(x)}).
\end{gather*}

Thus when $k = mn+1$ we obtain Conjecture \ref{conj:mnp1} from
the introduction. For general $(k,n)$ we lack an 
explicit description of the 
sheaf \(F_{k,n}\), but  computer experiments with the
combinatorial formula for the Poincare polynomial of the Hilbert 
scheme suggest:
 
\begin{conjecture} 
There exist $g_{r/n}(\lambda)$ such that 
\[ (at)^{-\mu} \sppa(X_{mn+r,n}) = 
\sum_{|\lambda|=n} \frac{g_{r/n}(\lambda)}{T(\CC^2)^{[n]}_\lambda}
 \cdot (1-T_1)(1-T_2) \cdot (\Lambda^* \tT')_\lambda  (\Lambda^* A\tT'^*)_\lambda
(\det \tT)_\lambda^m
\]
with the following properties:
\begin{gather}
\delta_n\cdot g_{r/n}(\lambda)\in\ZZ[T_1^{\pm1},T_2^{\pm1}],\\
\label{eq:sym} g_{r/n}(\lambda')=g_{r/n}(\lambda)|_{T_1=T_2,T_2=T_1}, \\
g_{r/n}(\lambda)=g_{(n-r)/n}(\lambda)|_{T_1=1/T_1,T_2=1/T_2} T_1^{\kappa(\lambda)} T_2^{\kappa(\lambda')}\\
g_{1/n}(\lambda)=\sum_{x\in \lambda} T_1^{l(x)} T_2^{a(x)},
\end{gather}
where $\delta_n=\prod_{0<i,j<n}(1-T_1^iT_2^j)$.
\end{conjecture}

The conjecture is confirmed by computer 
experiment up to $n<9$.
The last two formulas combined produce explicit formulas 
for the superpolynomial of  $T(n,mn-1)$ analogous to the one for $T(n,mn+1)$.

Our calculation of the stable superpolynomial in Proposition~\ref{prop:infty}, together with the symmetry of equation~\eqref{eq:sym} imply that 
$$ g_{r/n}( (1^n) )=T_1^{(n-1)(r-1)/2}\frac{T_1^n-1}{T_1-1},\quad  
g_{r/n}((n))=T_2^{(n-1)(r-1)/2}\frac{T_2^n-1}{T_2-1}.$$
Computer calculations suggest
the following formulas, which we have checked for \(n<8\):
$$ g_{r/n}(2,1^{n-2})=([n-r]_{T_1}+T_2[r]_{1/T_1}) T_1^{(n-1)(r-1)/2},$$
\begin{multline*} 
g_{r/n}(3,1^{n-3})=T_1^{\frac{(r-1)(n-1)}{2}}[(n-2r)\vee 0]_{T_1}+T_1^{\frac{(r-1)(n-3)}{2}}T_2[(n-r)\wedge r]_{T_1}+ \\
                   T_1^{\frac{(n-r-1)(n-3)}{2}}T_2^2[(n-r)\wedge r]_{1/T_1}+
		   T_1^{\frac{(n-r-1)(n-1)}{2}-n+2}T_2^3[(2r-n)\vee 0]_{1/T_1}
\end{multline*}
where $[n]_t:=(t^n-1)/(t-1)$,  $a\wedge b:=min(a,b)$ and $ a\vee b:=max(a,b)$.

For $n=5$,  the following data together with the symmetries discussed above is enough to  determine 
 $g_{r/5}(\lambda)$ for all $r$ and $\lambda$:

$$\begin{array}{c|c}
r&g_{r/5}(2^2,1) \\
\hline
&\\
1& (1+T_1+T_1^2)+(1+T_1)T_2\\
2&(T_1^2+T_1^3)+(T_1+T_1^2)T_2+T_2^2\\
3& T_1^4+(T_1^2+T_1^3)T_2+(T_1+T_1^2)T_2^2\\
4&(T_1^3+T_1^4)T_2+(T_1^2+T_1^3+T_1^4)T_2^2
\end{array}$$

The predicted answers become increasingly complicated
as $n$ grows.  
Below we show the answers in the case
$r=3$, $n=7$. Formulas for higher $n$ and different $r$'s are available by request to the
authors.
\begin{multline*}
g_{3/7}(22111)=T_1^2(T_1^7+2T_1^6+T_1^5(T_2+1)+T_1^4(2T_2+1)+2T_1^3 T_2+T_1^2(T_2^2+T_2)+T_1T_2^2+T_2^2)/(1+T_1^2)
\end{multline*}
\begin{multline*}
g_{3/7}(2221)=T_1^2(T_1^6T_2+T_1^5(T_2^2+T_2-1)+T_1^4(2T_2^2-1)
                           +T_1^3(T_2^3+T_2^2-T_2)\\
                 +T_1^2(T_2^3-T_2^2-T_2)-2T_1T_2^2+(T_2^4-T_2^3-T_2^2))/(T_2^2T_1-1)
\end{multline*}
\begin{multline*}
g_{3/7}(3211)=T_1(T_1^6(T_2+1)+T_1^5(T_2^2 + 3T_2+1)+T_1^4(T_2^3 + 4T_2^2 + 3T_2)+T_1^3(4T_2^3 + 6T_2^2 +T_2)\\
+T_1^2(T2^4 + 5T_2^3+3T_2^2)
+T_1(2T_2^4+3T_2^3+T_2^2)+T_2^4)/((T_2+1)(T_1^2+T_1+1))
\end{multline*}
\begin{multline*}
g_{3/7}(322)=(T_1^6(T_2^2+T_2-1)+T_1^5(T_2^3+2T_2^2)+T_1^4(T_2^4+2T_2^3-T_2^2-T_2)+T_1^3(T_2^4-T_2^3-2T_2^2)+\\
         T_1^2(T_2^5-2T_2^3-T_2^2)-T_1T_2^4)/(T_1^2T_2-1).
\end{multline*}
\begin{equation*}
g_{3/7}(4111)=T_1^3(T_2+1)+
T_1^2T_2+T_1(T_2^3+T_2^2+T_2) + T_2^3,
\end{equation*}

\begin{appendix}
\section{Combinatorics of HOMFLY homology, by Eugene Gorsky}

This appendix relates the combinatorics of the cells in the Hilbert scheme of a plane curve singularity with one Puiseaux pair to 
the existing results (\cite{hagl1},\cite{hagl2},\cite{EHKK}) on the combinatorics of diagonal harmonics and DAHA representations.
This connection justifies some of the conjectures made in \cite{g}, where a relation between $q,t$-Catalan numbers and torus knot homology
was proposed. 

As an outcome of this combinatorial study, the authors wrote a computer program computing the polynomials $\spp_{alg}(T(k,n))$ for any $k$ and $n$. These polynomials agree with 
all conjectured \cite{as,ch,mm} formulas for the superpolynomials of torus knots. The output of the program is available by request to the authors.

\subsection{Reformulation of Theorem \ref{thm:dimensions}}

Let us recall the setup of  Theorem \ref{thm:dimensions}.

One has a semigroup $\Gamma$ generated by two coprime integers $k$ and $n$,
a semigroup ideal $\ii$ with generators $i_1,\ldots,i_r$ and syzygies $s_1,\ldots,s_r$,
and a semigroup ideal $\jj=\ii\setminus \{i_1,\ldots i_m\}.$

Then by Theorem \ref{thm:dimensions} 
the dimension of the cell in the Hilbert scheme of a curve singularity with semigroup $\Gamma$ labelled by the ideal $\ii$
equals to
$$N(\ii)=
\sum_{a=1}^{r}\sharp(\Gamma_{>i_{a}}\setminus\ii)-
\sum_{a=1}^{r}\sharp(\Gamma_{>s_{a}}\setminus\ii).$$
while the dimension of the cell in the nested Hilbert scheme labelled by the pair $\ii\supset\jj$ equals to

$$N(\ii\supset\jj)=
\sum_{a=1}^{m}\sharp(\Gamma_{>i_{a}}\setminus\ii)+
\sum_{a=m+1}^{r}\sharp(\Gamma_{>i_{a}}\setminus\jj)-
\sum_{a=1}^{r}\sharp(\Gamma_{>s_{a}}\setminus\jj).$$

Let us relate the values of $N(\ii)$ and $N(\ii\supset \jj)$.

\begin{lemma}
\label{nn}
The following identity holds:
$$N(\ii\supset\jj)-N(\ii)=\sum_{a=m+1}^{r}\sum_{b=1}^{m}\chi(i_{a}<i_{b})-\sum_{a=1}^{r}\sum_{b=1}^{m}\chi(s_{a}<i_{b}).$$
\end{lemma}

\begin{proof}
Remark that 
$$N(\ii\supset\jj)-N(\mathfrak{i})=\sum_{a=m+1}^{r}[\sharp(\Gamma_{>i_{a}}\setminus\jj)-\sharp(\Gamma_{>i_{a}}\setminus\ii)]-\sum_{a=1}^{r}[\sharp(\Gamma_{>s_{a}}\setminus\jj)-\sharp(\Gamma_{>s_{a}}\setminus\ii)].$$
It rests to remark that for every $x\in \Gamma$ one has
$$\sharp(\Gamma_{>x}\setminus\jj)-\sharp(\Gamma_{>x}\setminus\ii)=\sum_{a=1}^{m}\chi(x<i_a).$$
.
\end{proof}

\begin{theorem}
\label{beta}
Let $$\beta_b(\ii):=\sum_{a=1}^{r}\chi(i_{a}<i_{b})-\sum_{a=1}^{r}\chi(s_{a}<i_{b}).$$
Then
\begin{equation}
\sum_{m=0}^{r}z^{m}q^{m\choose 2}\sum_{\sharp(\ii\setminus\jj)=m}q^{N(\ii\supset\jj)}=q^{N(\ii)}\prod_{b=1}^{r}(1+zq^{\beta_b(\ii)}).
\end{equation}
\end{theorem}

\begin{proof}
Remark that by Lemma \ref{nn}
$${m\choose 2}+{N(\ii\supset\jj)}=\sum_{a=1}^{m}\sum_{b=1}^{m}\chi(i_{a}<i_{b})+N(\ii)+\sum_{a=m+1}^{r}\sum_{b=1}^{m}\chi(i_{a}<i_{b})-\sum_{a=1}^{r}\sum_{b=1}^{m}\chi(s_{a}<i_{b})=$$
$$N(\ii)+\sum_{a=1}^{r}\sum_{b=1}^{m}\chi(i_{a}<i_{b})-\sum_{a=1}^{r}\sum_{b=1}^{m}\chi(s_{a}<i_{b})=N(\ii)+\sum_{b=1}^{m}\beta_{b}(\ii),$$
therefore
$$\sum_{m=0}^{r}z^{m}q^{m\choose 2}\sum_{\sharp(\ii\setminus\jj)=m}q^{N(\ii\supset\jj)}=
\sum_{m}\sum_{b_1,\ldots,b_m}z^{m}q^{N(\ii)+\beta_{b_1}(\ii)+\ldots+\beta_{b_{m}}(\ii)}=q^{N(\ii)}\prod_{b=1}^{r}(1+zq^{\beta_b(\ii)}).$$
\end{proof}

The geometric meaning of Theorem \ref{beta} is not known to the authors.
However, it is quite useful for the computations with the nested Hilbert scheme.

\begin{definition}
Let us call a number $x$ a $n$-generator of $\ii$, if $x\in \ii$, but $x-n\notin \ii$.
\end{definition}

\begin{lemma}
\label{betapgen}
The number $\beta_b(\ii)$ equals to the number of $n$-generators of $\ii$ on $]i_{b}-k,i_{b}]$.
\end{lemma}

\begin{proof}
Remark that if $u$ is a $n$-generator of $\ii$ then there exist a unique $(k,n)$-generator $i_{a}$ and
a unique syzygy $s_{c}$ such that $i_{a}=u-l_{1}k, \quad s_{c}=u+l_{2}k.$ 
All elements of the form $i_{a}+lk$ are $n$-generators for $0\le l<l_1+l_2.$
Now the statement follows from the equation
$$\chi(i_{a}<i_{b})-\chi(s_{a}<i_{b})=\sum_{l=0}^{l_1+l_2-1}\left(\chi(i_{a}+lk<i_{b})-\chi(i_{a}+(l+1)k<i_{b})\right)=$$ $$\sum_{l=0}^{l_1+l_2-1}\chi\left(i_a+lk\in ]i_{b}-k,i_{b}]\right).$$
\end{proof}

For semigroup ideals $\ii\supset \jj$ let us define $l(\ii)=\sharp(\Gamma\setminus \ii)$,
$m(\ii,\jj)=\sharp(\ii\setminus \jj).$  
Recall that the following formula is a corollary of the 
Theorem \ref{thm:dimensions}:

$$\usppa(T(p,q))=({a\over q})^{\mu(C)-1}\sum_{\ii\supset\jj}q^{2l}a^{2m}t^{m^2+2N(\ii\supset \jj)}.$$
Here $\mu(C)=(k-1)(n-1)$ is the Milnor number of the corresponding singularity.
 
Applying Theorem \ref{beta}, we obtain the following result: 

\begin{corollary}
\begin{equation}
\label{philb}
\usppa(T(k,n))=({a\over q})^{\mu(C)-1}\sum_{\ii}q^{2l(\ii)}t^{2N(\ii)}\prod_{m=1}^{r}(1+t^{2\beta_m(\ii)-1}a^2).
\end{equation}
\end{corollary}

\subsection{Compactified Jacobian}

In what follows we need some detailed information on the structure of the compactified Jacobian of a singularity with semigroup $\Gamma$,
that will allow us to give a conjectural reformulation of the results of \cite{MY} and \cite{MS}.

\begin{definition}(\cite{Pi}) Let $\Delta$ be a $0$-normalized $\Gamma$--semi-module, i.e. $0\in \Delta\subset\mathbb Z_{\ge 0},$ and $\Delta + \Gamma\subset \Delta.$ The dimension of $\Delta$ is defined as
$$\dim \Delta=\sum_{j=0}^{n-1}\sharp\left([a_j, a_j + k[\setminus \Delta\right) .$$
\noindent where $(0=a_0<a_1<\ldots<a_{n-1})$ are the $n$--generators of $\Delta.$
\end{definition}

\begin{theorem}(\cite{Pi}) 
\label{pi}
The Jacobi factor of the singularity with semigroup $\Gamma$ admits the natural cell decomposition with affine cells $C_{\Delta}$.
The cells are parametrised by the 0-normalized $\Gamma$--semi-modules $\Delta$, and the dimension of the cell equals to
$\dim \Delta$.  
\end{theorem}

We will parametrise these cells by the certain Young diagrams. Consider the $n\times k$ rectangle $R$ and draw the diagonal from the top-left to the bottom-right corner.

\begin{definition}
 Let $R_+\subset R$ be the subset consisting of boxes which lie under the left-top to right-bottom diagonal.
\end{definition}

Label the boxes of $R$ and around with integers,
 so that the shift by $1$ up subtracts
$n,$ and the shift by $1$  to the right subtracts $k.$ 
We normalize these numbers so that
$kn$ is in the box $(0,0)$ (note that this box is not in
 the rectangle $R,$ as we start enumerating
boxes from $1$). In other words, the numbers are
 given by the linear function $f(x,y)=kn-kx-ny.$
One can see that the labels of the boxes of $R_+$ are positive, 
while all other numbers in
$R$ are negative. Moreover, numbers in the boxes of $R_+$ are exactly the numbers from the
complement $\mathbb Z_{\ge 0}\backslash\Gamma,$ and each such number appears only once in $R_+.$ In
particular, the area of $R_+$ is equal to $\delta=\frac{(k-1)(n-1)}{2}.$

\begin{definition}(\cite{GM})
For a  $0$-normalized $\Gamma$--semi-module $\Delta$, let $D(\Delta)$ denote the set of boxes with labels belonging to 
$\Delta\setminus\Gamma$.
\end{definition}

\begin{definition}(\cite{lowa})
Let $D$ be a Young diagram, $c\in D$. Let $a(c)$ and $l(c)$ denote the lengths of arm and leg for $c$.
For each real nonnegative $x$  define
$$h^{+}_{x}(D)=\sharp\left\{c\in D~~\vline~~{a(c)\over l(c)+1}\le x< {a(c)+1\over l(c)}\right\}.$$
\end{definition}

The following theorem is the main result of \cite{GM}.

\begin{theorem}
\label{piontcell}
The dimensions of cells can be expressed through the $h^{+}$ statistic:
$$\dim C_{\Delta}=\frac{(k-1)(n-1)}{2}-h^{+}_{\frac{n}{k}}(D(\Delta)).$$
\end{theorem}

\begin{conjecture}
\label{HilbJClow}
One can match the following generating functions for the Hilbert scheme of points and the compactified Jacobian:
\begin{equation}
\usppa(a=0,q,t)=\sum_{\mathfrak{i}}q^{2l(\mathfrak{i})}t^{2N(\mathfrak{i})}={1\over 1-q^2}\sum_{D}q^{2|D|+2h^{+}_{\frac{n}{k}}(D)}t^{2|D|},
\end{equation}
\begin{equation}
\sppa(a=0,q,t)=\sum_{D}q^{2|D|+2h^{+}_{\frac{n}{k}}(D)}t^{2|D|}.
\end{equation}
\end{conjecture} 

\begin{remark}
This conjecture is expected to be the combinatorial counterpart of the generalized Macdonald formula (\ref{eq:myms}),
proved in \cite{MY} and \cite{MS}. Namely, $|D(\Delta)|=\sharp(\Delta\setminus\Gamma)$
is expected to be related to the perverse filtration on the cohomology of the compactified Jacobian.
\end{remark}

To formulate the analogous conjecture for the nested Hilbert scheme, we have to define the analogues of $\beta$-statistic for the admissible diagrams.
Roughly speaking, we consider the complement to a diagram $D$ as an ideal whose generators correspond to the SE-corners of $D$ and syzygies correspond to
the ES corners of $D$. Remark that to get the corresponding semigroup ideal one has to replace a number $x$ by $kn-k-n-x$, thus reversing the order.

\begin{definition}
Consider a diagram $D$ corresponding to a semigroup module $\Delta$.
Let $P_m$ denote the numbers in the SE corners, $Q_i$ denote the numbers in the $ES$ corners. Then
$$\beta(P_m)=\sum_{i}\chi(P_i>P_m)-\sum_{i}\chi(Q_i>P_m).$$
\end{definition}

\begin{example}
\label{ex3221}
Consider a semigroup generated by $5$ and $6$, and a module 
$$\Delta=\{0,1,2,5,6,\ldots\}.$$
Its diagram has a form:

\ifhavetikz
\begin{tikzpicture}
\draw  (0,0)--(0,4);
\draw  (1,0)--(1,3);
\draw  (2,0)--(2,2);
\draw  (3,0)--(3,2);
\draw  (4,0)--(4,1);
\draw  (0,0)--(5,0);
\draw  (0,1)--(4,1);
\draw  (0,2)--(3,2);
\draw  (0,3)--(1,3);
\draw (0.5,0.5) node {$19$};
\draw (0.5,1.5) node {$14$};
\draw (0.5,2.5) node {$9$};
\draw (1.5,0.5) node {$13$};
\draw (1.5,1.5) node {$8$};
\draw (2.5,0.5) node {$7$};
\draw (3.5,0.5) node {$1$};
\draw (2.5,1.5) node {$2$};
\draw (0.5,3.5) node {$\bf{4}$};
\draw (1.5,2.5) node {$\bf{3}$};
\draw (4.5,0.5) node {$\bf{-5}$};
\draw (1.5,3.5) node {$\bf{-2}$};
\draw (3.5,2.5) node {$\bf{-9}$};
\draw (4.5,1.5) node {$\bf{-10}$};
\draw (3.5,1.5) node {$\bf{-4}$};
\end{tikzpicture}
\fi

We have $$\{P_i\}=\{-5,-4,3,4\},\quad \{Q_j\}=\{-10,-9,-2\}.$$
Therefore $$\beta(-5)=3-1=2,\quad \beta(-4)=2-1=1,\quad \beta(3)=1,\quad \beta(4)=0.$$
\end{example}

\begin{conjecture} \label{nnp1}
\label{HilbJCfull}
One can match the following generating functions for the Hilbert scheme of points and the compactified Jacobian:
\begin{equation}
\usppa(a,q,t)=\sum_{\ii}q^{2l(\ii)}t^{2N(\ii)}\prod_{m=1}^{r}(1+t^{2\beta_{m}(\ii)-1}a^2)={1+a^2t\over 1-q^2}\sum_{D}q^{2|D|+2h^{+}_{\frac{n}{k}}(D)}t^{2|D|}\prod_{m=1}^{r}(1+a^2q^{-2\beta(P_m)}t),
\end{equation}
\begin{equation}
\label{JCfull}
\sppa(a,q,t)=\sum_{D}q^{2|D|+2h^{+}_{\frac{n}{k}}(D)}t^{2|D|}\prod_{m=1}^{r}(1+a^2q^{-2\beta(P_m)}t).
\end{equation}
\end{conjecture}

\subsection{Comparison of combinatorial statistics}

Let us denote the right hand side of the Conjecture \ref{nnp1} by $\spp_{\mbox {\rm \tiny DAHA}}(T(n,mn+1))$.
The bivariate polynomial $$C_{n}^{(m)}(q,t):=\spp_{\mbox {\rm \tiny DAHA}}(T(n,mn+1),a=0)$$ was introduced by A. Garsia and M. Haiman in \cite{GH} in connection
with the conjectures of \cite{haiman} on the structure of the module of diagonal harmonics, eventually proved in \cite{haiman2}.
In the special case the polynomials $C_n(q,t):=C_{n}^{(1)}(q,t)$ are called the $q,t$-Catalan numbers.

In case where $(p,q)=(n,n+1)$, the statistics $h^{+}_{n+1\over n}(D)$ is also called $dinv(D)$.
Therefore the Theorem \ref{piontcell} can be reformulated for this case as
$$\dim\Delta_{D}={n\choose 2}-dinv(D).$$

\begin{theorem}(\cite{haglund}) 
The $q,t$-Catalan numbers admit the following description:
\begin{equation}
\label{cn}
C_n(q, t)=\sum_{D}q^{dinv(D)}t^{{n\choose 2}-|D|}.
\end{equation}
\end{theorem}
Modulo Conjecture \ref{HilbJClow} we obtain the identity $$\sppa(T(n,n+1),a=0)=\spp_{\mbox {\rm \tiny DAHA}}(T(n,n+1),a=0).$$
 
It has been conjectured in \cite{loehr}, that the analogue of (\ref{cn}) 
holds for $(n,mn+1)$ case:
\begin{equation}
C^{(m)}_n(q, t)=\sum_{D}q^{h^{+}_{mn+1\over n}(D)}t^{m{n\choose 2}-|D|}.
\end{equation}
Modulo this conjecture and Conjecture \ref{HilbJClow}, the identity $$\sppa(T(n,mn+1),a=0)=\spp_{\mbox {\rm \tiny DAHA}}(T(n,mn+1),a=0)$$ holds as well.

The combinatorial statistics for higher $a$-levels for $(n,n+1)$ case were  proposed in \cite{EHKK}.
For their definition we will use a combinatorial bijection on Dyck paths, described in \cite{haglund}.

\begin{definition}
Let $\Delta$ be a $(n,n+1)$-semimodule, and let $a_0,\ldots,a_{n-1}$ be its $n$-generators.
Define a Young diagram $G(\Delta)$ with columns $g(a_0),\ldots, g(a_{n-1}).$
\end{definition}

The following result describes the properties of the map $G$.

\begin{theorem}(\cite{GM})
The following statements hold in $(n,n+1)$ case:
\begin{itemize}
\item[1.]{For any $\Delta$ the diagram $G(\Delta)$ is below the diagonal.}
\item[2.]{The correspondence between $\Delta$ and $G(\Delta)$ is bijective.}
\item[3.]{This bijection coincide with the bijection from \cite{haglund} exchanging 
(dinv,area) statistics with (area,bounce) statistics.}
\end{itemize}
\end{theorem}

\begin{remark}
We conjecture that the map $G$ is bijective in general.
\end{remark}

\begin{definition}
Let $b_i=n-1-i-g(a_i)$ be the number of cells in column $i$ between the diagram $G(\Delta)$ and the diagonal.
\end{definition}

The following theorem was conjectured in \cite{EHKK} and proved in \cite{hagl2}.

\begin{theorem}(\cite{hagl2})
\label{hookcomb}
$$\spp_{\mbox {\rm \tiny DAHA}}(T(n,n+1))=\sum_{D}q^{dinv(D)}t^{{n\choose 2}-|D|}\prod_{b_i>b_{i+1}}(1+a^2q^{-b_i}t).$$
\end{theorem}

Two following lemmas show that this combinatorial formula is equivalent to the equation \ref{JCfull}.

\begin{lemma}
Let $a_i$ and $a_{i+1}$ be two consecutive $n$-generators of a $\Gamma$-semimodule $\Delta$ with a diagram $D$.
Then the following statements are equivalent:
$$b_{i}>b_{i+1}\Leftrightarrow g(a_i)=g(a_{i+1})\Leftrightarrow a_i+1\in \Delta \Leftrightarrow a_i-n\quad \mbox{\rm is an SE corner of}\quad  D.$$
\end{lemma}

\ifhavetikz
\begin{tikzpicture}
\draw [thick] (1,0)--(0,0)--(0,1);
\draw (0.7,0.5) node {$a_i-n$};
\draw (0.5,-0.5) node {$a_i$};
\draw (-0.7,0.5) node {$a_i+1$};
\end{tikzpicture}
\fi

\begin{proof}
Since $g(a_i)\ge g(a_{i+1})$, one can check that $$b_i>b_{i+1}\quad \Leftrightarrow\quad g(a_i)=g(a_{i+1}).$$

Let $c$ be the maximal number such that $[a_i,c]\in \Delta$, suppose that $a_i<c<a_{i+1}$.
Since $c$ is not a $n$-generator, $c-n\in \Delta\Rightarrow c+1\in \Delta$.
Contradiction, therefore $c=a_i$ or $c=a_{i+1}.$

In the first case let $d$ be the maximal number such that $[a_i+1,d]\cup \Delta=\emptyset$, 
then $$d+1\in \Delta\Rightarrow d+n+1\in \Delta\Rightarrow g(a_{i+1})>g(a_i).$$
Contradiction, therefore $c=a_{i+1}$ and $[a_i,a_{i+1}]\subset \Delta$.
\end{proof}

\begin{lemma}
\label{betab}
The following relation holds:
$$\beta(a_i-n)=b_i.$$
\end{lemma}

\begin{proof}
By definition, $g(a_i)=\sharp([a_i,a_i+n[\setminus \Delta)$, so $$n-g(a_i)=\sharp([a_i,a_i+n[\cap \Delta).$$
For every $j<i$ there exists a unique element of the form $a_j+kn$ on $[a_i,a_i+n]$, which are not $n$-generators.
The remaining  $n-1-i-g(a_i)$ elements are $n$-generators of $\Delta$, hence the desired relation follows from Lemma \ref{betapgen}.
\end{proof}

\begin{example}
Let us return to the Example \ref{ex3221}:

\ifhavetikz
\begin{tikzpicture}
\draw  (0,0)--(0,4);
\draw  (1,0)--(1,3);
\draw  (2,0)--(2,2);
\draw  (3,0)--(3,2);
\draw  (4,0)--(4,1);
\draw  (0,0)--(5,0);
\draw  (0,1)--(4,1);
\draw  (0,2)--(3,2);
\draw  (0,3)--(1,3);
\draw (0.5,0.5) node {$19$};
\draw (0.5,1.5) node {$14$};
\draw (0.5,2.5) node {$9$};
\draw (1.5,0.5) node {$13$};
\draw (1.5,1.5) node {$8$};
\draw (2.5,0.5) node {$7$};
\draw (3.5,0.5) node {$1$};
\draw (2.5,1.5) node {$2$};
\draw (0.5,3.5) node {$\bf{4}$};
\draw (1.5,2.5) node {$\bf{3}$};
\draw (4.5,0.5) node {$\bf{-5}$};
\draw (1.5,3.5) node {$\bf{-2}$};
\draw (3.5,2.5) node {$\bf{-9}$};
\draw (4.5,1.5) node {$\bf{-10}$};
\draw (3.5,1.5) node {$\bf{-4}$};
\end{tikzpicture}
\fi

We showed that $$\beta(-5)=3-1=2,\quad \beta(-4)=2-1=1,\quad \beta(3)=1,\quad \beta(4)=0$$
The $5$-generators corresponding to the internal corners are $0,1,8$.
Since $g(0)=g(1)=g(2)=2$, the  diagram $G(\Delta)$ looks as

\ifhavetikz
\begin{tikzpicture}
\draw  (0,0)--(0,5);
\draw [dashed] (1,0)--(1,4);
\draw [dashed] (2,0)--(2,3);
\draw [dashed] (3,0)--(3,2);
\draw [dashed] (4,0)--(4,1);
\draw (0,0)--(5,0);
\draw [dashed] (0,1)--(4,1);
\draw [dashed] (0,2)--(3,2);
\draw [dashed] (0,3)--(2,3);
\draw [dashed] (0,4)--(1,4);
\draw (0,5) -- (5,0);
\draw [thick] (0,0)--(0,2) -- (3,2) -- (3,0)--(0,0);
\end{tikzpicture}
\fi

On diagram  $G(\Delta)$ we count $b(0)=2,\quad b(1)=1,\quad b(8)=1$.

Finally, since $|D|=8, h^{+}_{6/5}=10-|G(D)|=4$, the contribution of this $\Gamma$-semimodule in (\ref{JCfull})
equals to $q^{24}t^{16}(1+a^2q^{-2}t)^2(1+a^2q^{-4}t)$.
\end{example}

 \end{appendix}

\end{document}